\documentclass[onefignum,onetabnum]{siamart220329}

\usepackage{amsmath}%
\usepackage{amsfonts}%
\usepackage{amssymb}%
\usepackage{commath}
\usepackage{graphicx}
\usepackage{dsfont}
\usepackage{pdfcomment}
\usepackage{tikz}
\usepackage{tikz-3dplot}
\usepackage{soul}
\usepackage{enumerate}
\usepackage{cancel}

\usepackage{hyperref}
\usepackage{cleveref}
\usepackage{esvect}
\usepackage{subcaption} % for use of sub-figures

\def\C{\mathbb{C}}

\def\R{\mathbb{R}}

\usepackage{array}
\newcolumntype{C}[1]{>{\centering\let\newline\\\arraybackslash\hspace{0pt}}m{#1}}

\newcommand\restr[2]{{% we make the whole thing an ordinary symbol
  \left.\kern-\nulldelimiterspace % automatically resize the bar with \right
  #1 % the function
  \vphantom{\big|} % pretend it's a little taller at normal size
  \right|_{#2} % this is the delimiter
  }}

\graphicspath{ {/images/} }

\usepackage[ruled,linesnumbered,vlined,algo2e,resetcount]{algorithm2e}
\newsiamremark{remark}{Remark}
\newsiamremark{hypothesis}{Hypothesis}
\crefname{hypothesis}{Hypothesis}{Hypotheses}
\newsiamthm{claim}{Claim}

\title{Surface parameterization via optimization of relative entropy and quasiconformality\thanks{Submitted to the editors DATE.
\funding{This work was supported in part by the HKRGC GRF under project ID 14309125 (to Lok Ming Lui).}}}

\author{Zhipeng Zhu\thanks{
              Department of Mathematics, The Chinese University of Hong Kong
              (\email{zpzhu@math.cuhk.edu.hk}).}
           \and
           Lok Ming Lui\thanks{
              Department of Mathematics, The Chinese University of Hong Kong
              (\email{lmlui@math.cuhk.edu.hk}).}
}

\ifpdf
\hypersetup{ pdftitle={Guide to Using  SIAM'S \LaTeX\ Style} }
\fi

\begin{document}
\maketitle

\begin{abstract}
We propose a novel method for parameterizations of triangle meshes by finding an optimal quasiconformal map that minimizes an energy consisting of a relative entropy term and a quasiconformal term. By prescribing a prior probability measure on a given surface and a reference probability measure on a parameter domain, the relative entropy evaluates the difference between the pushforward of the prior measure and the reference one. The Beltrami coefficient of a quasiconformal map evaluates how far the map is close to an angular-preserving map, i.e., a conformal map. By adjusting parameters of the optimization problem, the optimal map achieves a desired balance between the preservation of measure and the preservation of conformal structure. To optimize the energy functional, we utilize the gradient flow structure of its components. The gradient flow of the relative entropy is the Fokker-Planck equation, and we apply a finite volume method to solve it. Besides, we discretize the Beltrami coefficient as a piecewise constant function and apply the linear Beltrami solver to find a piecewise linear quasiconformal map.
\end{abstract}


%% ------------------------------------------------------------------
%% END HEADER
%% ------------------------------------------------------------------

\section{Introduction}
\label{sec:intro}
In contemporary applied mathematics and computer science, triangulated surfaces play an important role in many fields, such as numerical PDEs, texture map and 3D reconstruction in computer graphics, brain mapping in medical imaging, and 3D object detection and classification in computer vision. To ensure the convergence and stability of many algorithms, the quality of a triangle mesh is highly valued. We observe that researchers often give the following two requirements for triangle meshes. The first is about the distribution of vertex points of a triangulation. For example, to secure the accuracy of numerical PDE algorithms, researchers often require that the vertex points be dense enough in some region. Another requirement is that the triangle faces should be close enough to equilateral triangles, which is also essential for the convergence of many algorithms. In this work, we propose a novel method to address the two requirements in a unified framework. That is, given a surface $\mathcal{S}$ and a parameter domain $\Omega$, we minimize over all bijective maps $f:\mathcal{S}\to\Omega$, a functional involving the relative entropy and quasiconformality. The relative entropy controls the vertex distribution of the mesh while the quasiconformality measures the angular distortion of the map $f$.

To mathematically define the relative entropy, we need to prescribe a probability measure $\gamma$ on $\mathcal{S}$ and a reference probability measure $\nu$ on $\Omega$. For example, if we hope to compute an area-preserving parameterization, we can define the measure of each triangle face as its area divided by the total area and simply take $\gamma$ to be Lebesgue divided by a constant. To compare the difference between the pushforward measure $f_* \gamma$ and $\nu$, one approach is to find the Wasserstein distance $W_2(f_*\gamma,\nu)$. However, Wasserstein distances can be quite difficult to compute and optimize. Thus, we turn to the relative entropy $\mathcal{H}(f_*\gamma|\nu)$, an alternative that is easier to compute and minimize. To find the optimal map $f$, we utilize the fact that the Fokker-Planck equation is the Wasserstein gradient flow of the relative entropy functional in the space of probability measures~\cite{jordan1998variational}. Suppose we have an initial map $f_0$ such that $(f_0)_*\gamma$ has density $\rho_0$. The Fokker-Planck equation will give us a time-dependent density $\rho_t$ that approaches the density of the reference measure as $t\to\infty$. This $\rho_t$, together with the continuity equation, will give us a time-dependent vector field $v_t$. Integrating $v_t$ from the initial map $f_0$ will give us a time-dependent map $f_t$ such that $(f_t)_*\gamma$ has density $\rho_t$. Thus, we can expect that the densities can be nearly matched by $f_t$ when $t$ is large.

The quasiconformal maps are the generalization of conformal maps. A conformal map maps infinitesimal circles to circles while a quasiconformal map maps infinitesimal circles to ellipses with bounded eccentricity. Every quasi-conformal map $f:\C\to\C$ is associated with a complex-valued function $\mu$ with $|\mu(z)|<1$, which is called its Beltrami coefficient. The magnitude of Beltrami coefficient at a point measures the local angular distortion. If the Beltrami coefficient is zero at a point, then the quasiconformal map will preserve angles in the infinitesimal level. Besides, by the measurable Riemann mappping theorem~\cite{ahlfors1960riemann}, there is a $1$-$1$ correspondence between measurable Beltrami coefficients and normalized orientation-preserving homeomorphisms in the complex plane. As a result, a quasiconformal map can be recovered from its Beltrami coefficient by solving the Beltrami equation. To evaluate the global angular distortion, we propose to compute the $L^2$-norm of the Beltrami coefficient. The $L^2$-norm of the gradient of the Beltrami coefficient can also be involved to enhance the smoothness of a map. To reduce the angular distortion of an arbitrary map $f$, we can first calculate its Beltrami coefficient. Solving the Beltrami equation with the reduced and smoothed Beltrami coefficient will then give us an updated map with reduced angular distortion.

The energy functional we hope to minimize consists of the relative entropy, the $L^2$-norm of the Beltrami coefficient $\mu$, and the $L^2$-norm of the gradient of $\mu$. It is difficult to minimize it directly, since the feasible set consists of all piecewise linear bijective maps. But We notice that each term has a gradient flow in some appropriate space. So, we propose a splitting method to optimize the energy. Briefly speaking, we alternate between one term of the energy to get an updated bijective map each time and for another. To be more precise, we need to choose time steps $t_1$, $t_2$, and $t_3$. In each iteration, we let the current map $f$ to evolve under the gradient flow of the relative entropy for time $t_1$ first, then under the gradient flow of $L^2$-norm of $\mu$ for time $t_2$, and finally under the gradient flow of $L^2$-norm of $\nabla\mu$ for time $t_3$. By manipulating the values of $t_1$, $t_2$, and $t_3$, we can emphasize which term is more important and which is less. When we explicitly discretize these gradient flows, and $t_3$, we may further pick time scales that are smaller than $t_1$, $t_2$, and $t_3$ to enhance the accuracy and stability.

\subsection{Related Works}
Surface parameterization, which focuses on mapping a complicated surface in $\R^3$ to a standard domain in $\R^2$ to facilitate subsequent manipulation and processing, has attracted much research attention in recent decades. Profound research works include discrete surface harmonic maps~\cite{pinkall1993computing,remacle2010high}, shape-preserving paramerization~\cite{floater1997parametrization},  discrete conformal maps~\cite{levy2023least,desbrun2002intrinsic,gu2004genus}, feature-based parameterization~\cite{zhang2005feature}, locally authalic maps~\cite{desbrun2002intrinsic}, large diffeomorphic
distance metric mapping (LDDMM)~\cite{beg2005computing,zhang2017frequency}, discrete surface Ricci flow~\cite{jin2008discrete}, Lie advection~\cite{zou2011authalic}, optimal transport maps~\cite{zhao2013area,cui2019spherical}, and density-equalizing maps~\cite{choi2018density}. 

Our method for optimizing the relative entropy relies on the fact that its Wasserstein gradient flow is the Fokker-Planck equation. In recent years, several interesting and successful numerical methods have been proposed to approximate and compute Wasserstein gradient flows, including finite elements~\cite{burger2009mixed}, finite volume method~\cite{bessemoulin2012finite}, the steepest descent scheme~\cite{blanchet2008convergence}, Lagrangian maps~\cite{during2010gradient,junge2017fully}, evolving diffeomorphisms~\cite{carrillo2010numerical}, Galerkin method~\cite{sun2018discontinuous}, computational geometry method~\cite{benamou2016discretization}, blob method~\cite{carrillo2019blob}, and primal-dual method~\cite{carrillo2022primal}.

Since quasiconformal maps are less restrictive than conformal maps, there has been an increasing interest in computational quasiconformal geometry and its applications in the recent year. Existing methods for computing quasiconformal maps include auxiliary metric~\cite{zeng2012computing}, quasi-Yamabe flow~\cite{zeng2012computing}, extermal quasiconformal map~\cite{weber2012computing}, holomorphic Beltrami flow~\cite{lui2012optimization}, linear Beltrami solvers~\cite{lui2013texture,qiu2019computing}, and the welding method~\cite{zhu2022parallelizable}. Applications of computational quasiconformal geometry include image and surface registration~\cite{lam2014landmark,qiu2020inconsistent}, and shape analysis~\cite{choi2020tooth,choi2020shape}. 

\section{Background and Formulation}
\label{sec:background}
\subsection{Some Preliminaries}
\subsubsection{Relative Entropy and Fokker-Planck Equation}
The relative entropy functional can be faithfully defined on the space of $\mathcal{P}_2^{ac}(\R^d)$ of absolutely continuous probability measures on $\R^d$ equipped with the Wasserstein-2 metric. For any $\mu,\nu\in\mathcal{P}_2^{ac}(\R^d)$, there always exists a transport map $t:\R^d\to\R^d$ such that $t_*\mu=\nu$. Thus, their Wasserstein distance can be given by
\begin{equation}
    W_2(\mu,\nu) = \min\bigg\{ \int_{\R^d}|x-t(x)|^2\,d\mu(x):t_*\mu=\nu \bigg\}
\end{equation}
Although relative entropy is not a metric, it gives another way to measure the distance between two probability measures $\mu,\nu\in\mathcal{P}_2^{ac}(\R^d)$. Mathematically, it is defined as
\begin{equation}
\label{eq:relative_entropy}
    \mathcal{H}(\mu|\nu) = \int_{\R^d}\dfrac{d\mu}{d\nu}\log\bigg(\dfrac{d\mu}{d\nu}\bigg)\,d\nu
\end{equation}
where the density $d\mu/d\nu$ is the Radon-Nikodym derivative of $\mu$ with respect to $\nu$. Similarly to the Wasserstein interpolation, the gradient flow of relative entropy leads to another interpolation between two probability measures. Let us denote the density of $\mu$ and $\nu$ by $\zeta$ and $u$ respectively. The Jordan-Kinderlehrer-Otto scheme states that the Fokker-Planck equation
\begin{equation}
\label{eq:fokker-planck}
    \partial_t\rho_t - \nabla\cdot(\nabla\rho_t+\rho_t\nabla V) = 0\ \ \ \text{in }\R^d\times [0,\infty)
\end{equation}
where $V$ satisfies $u=e^{-V}$, is the gradient flow of relative entropy. If we set $\rho_0=\zeta$, the density $\rho_t$ will converge to $u$ as $t\to\infty$. While Fokker-Planck equation gives us the flow of the density, we are especially concerned with the transport map in this article. This map can be found by solving the continuity equation
\begin{equation}
    \partial_t \rho_t + \nabla\cdot(v_t\rho_t) = 0 \ \ \ \ \text{in }\R^d\times [0,T]
\end{equation}
to get a time-dependent vector field $v:(x,t)\to v_t(x)\in\R^d$. Integrating $v_t$ from time $0$ to time $T$ produces the transport map that pushforwards $\rho_0$ to $\rho_T$. 

The above is a sketch of the mathematical theories that motivate our algorithm. We do not dig into the details, such as the precise definition of gradient flows in metric spaces. Interested readers are highly recommended to read the marvelous book~\cite{ambrosio2008gradient} written by Ambrosio, Gigli, and Savaré.
 
\subsubsection{Planer quasi-conformal theory}
\label{subsec:qc2d}
Quasi-conformal maps are generalizations of conformal maps in the sense that they permit bounded distortion of angles locally. Several equivalent definitions characterize quasi-conformal maps. Here, we give an analytic definition.
\begin{definition}
    Let $D$ and $D'$ be two domains in $\C$, and set $k:=\dfrac{K-1}{K+1}$ for some $K\geq 1$. An orientation-preserving homeomorphism $f:D\to D'$ is $K$-quasiconformal if $f\in W_{loc}^{1,2}(\Omega)$ and its partial derivatives satisfy
    \begin{equation}
    \label{eq:qc-definition}
        \left|\dfrac{\partial f}{\partial \bar{z}}\right| \leq k \left| \dfrac{\partial f}{\partial z} \right|
    \end{equation}
    for almost every $z\in D$.
\end{definition}
The most intriguing feature of quasiconformal maps in the complex plane is that we could reduce this definition to the Beltrami equation associated with the complex dilatation $\mu:D\to\C$ called Beltrami coefficients. Their correspondence is established by the mapping theorem of Bers and Ahlfors~\cite{ahlfors1960riemann}.
\begin{theorem}[Measurable Riemann Mapping Theorem]
\label{thm:measurable_riemann}
    Let $\abs{\mu(z)}\leq k<1$ for almost every $z\in \C$. Then there is a solution $f:\hat{\C}\to\hat{\C}$ to the Beltrami equation
    \[\dfrac{\partial f}{\partial \bar{z}} = \mu(z)\dfrac{\partial f}{\partial z}\ \ \ \text{for almost every }z\in\C\]
    which is a $K$-quasi-conformal homeomorphism normalized by the three conditions
    \[f(0)=0;\ \ \ f(1)=1;\ \ \ f(\infty)=\infty\]
    Furthermore, the normalized solution $f$ is unique.
\end{theorem}
A similar result can be obtained when $\C$ is replaced by the disk or by arbitrary domains. A direct consequence of the mapping theorem is that one can explicitly find a quasiconformal map by solving the Beltrami equation, whose parameters have a very weak regularity assumption. Indeed, quasiconformal maps are closely related to elliptic partial differential equations~\cite{astala2008elliptic}. Suppose $f=u+iv$ is quasiconformal with Beltrami coefficients $\mu=\rho+i\tau$, where $u,v,\rho$, and $\tau$ are real. Define
\begin{equation}
    A = \frac{1}{1-\abs{\mu}^2}
    \begin{pmatrix}
    (\rho-1)^2 + \tau^2 & -2\tau \\
    -2\tau & (1+\rho)^2 + \tau^2
    \end{pmatrix}.
\label{eqt:dilation_matrix}
\end{equation}
It can be shown that the Beltrami equation is equivalent to
\begin{equation*}
(Df)^t Df = J_f A^{-1}
\end{equation*}
When $\mu$ is supported on a compact set $\Omega$, $f|_{\Omega}$ can be found be minimizing 
\begin{equation}
\label{eq:qc_energy_1}
    \dfrac{1}{2}\int_{\Omega}\,\norm{A^{1/2}\nabla u}^2 + \norm{A^{1/2}\nabla v}^2 dxdy - \int_{\Omega} J_f \, dxdy 
\end{equation}
with a prescribed landmark condition $f(p_1)=q_1$ and $f(p_2)=q_2$ for $p_1,p_2\in\Omega$. If we prescribe a Dirichlet boundary condition $g:\partial\Omega\to\partial\Omega'$, the Beltrami equation usually has no solution. However, we could still aim to find an orientation-preserving homeomorphism minimizing the energy
\begin{equation}
\label{eq:qc_energy_2}
    \dfrac{1}{2}\int_{\Omega}\,\norm{A^{1/2}\nabla u}^2 + \norm{A^{1/2}\nabla v}^2 dxdy
\end{equation}
to get a least-square quasi-conformal map. With some abuse of notation, we let $\Gamma$ denote either a landmark condition or a Dirichlet condition. In the remaining article, we denote the quasiconformal map associated with $\mu$ by $S_{\Gamma}(\mu)$.

\subsection{Formulation of our method}
\label{sec:formulation}
The primary goal of our method is to find a quasiconformal homeomorphism from an initial surface $\mathcal{S}_1$ to a target surface $\mathcal{S}_2$ minimizing an energy functional involving relative entropy and quasiconformality. We give a theoretical formulation here and will talk about the discretizations later. Suppose $\mathcal{S}_1$ is a Riemann surface equipped with an absolutely continuous probability measure $\gamma$. We do not place too many restrictions on the topology of $\mathcal{S}_1$, but we mainly work on Riemann surfaces which admit a metric of constant curvature $0$ in this article. The target surface $\mathcal{S}_2$ is a flat Riemann surface that is homeomorphic to $\mathcal{S}_1$ and is equipped with another absolutely continuous probability measure $\nu$. Based on the topology of $\mathcal{S}_1$, $\mathcal{S}_2$ can be chosen as the unit disk, a circular domain, or the flat torus $\R^2/\mathbb{Z}^2$.

Over all quasiconformal maps $f:\mathcal{S}_1\to\mathcal{S}_2$, we aim to minimize
\begin{equation}
\label{eq:energy_measure_qc0}
    E(f) = \mathcal{H}(f_*\gamma|\nu) + \alpha\norm{\mu_f}_{L^2}^2
\end{equation}
where $\mathcal{H}$ denotes the relative entropy, $\mu_f$ is the Beltrami coefficient associated with $f$, and $\alpha>0$ is prescribed. We need to be cautious about the second term. Although the Beltrami coefficient is independent of the choice of charts and therefore well defined, its $L^2$ norm actually depends on the choice of charts. To make $\norm{\mu_f}_{L^2}^2$ well-defined, one way is to fix a realization of the conformal structure of $\mathcal{S}_1$. That is, we need to find a conformal map $\varphi:\mathcal{S}_1\to \Omega\subset\C$ and integrate the function $\norm{\mu_f}^2$ on $\Omega$. Another way to eliminate the ambiguity is to integrate $\norm{\mu_f}^2$ on $\mathcal{S}_1$ with respect to the volume form or a given surface measure. In our formulation, we choose to adopt the first approach. Now, suppose that $f_*\gamma$ has density $\rho$ and $\nu$ has density $u=e^{-V}$, and write $\mu_f=\sigma+i\tau$ for some real-valued function $\sigma$ and $\tau$ on $\Omega$. Then, we can express the energy as 
\begin{align}
    E(f) &= \int_{\mathcal{S}_2} \rho \log\big(\dfrac{\rho}{u}\big)\,dx + \alpha\int_{\Omega} (|\sigma|^2 + |\tau|^2)\,dx \\
    &= \int_{\mathcal{S}_2} \rho\log\rho\,dx - \int_{\mathcal{S}_2} \rho V\,dx + \alpha\int_{\Omega} (|\sigma|^2 + |\tau|^2)\,dx
\end{align}
Notice that the integration domain of the relative entropy is $\mathcal{S}_2$ while the integration domain of the $L^2$-norm of $\mu_f$ is $\Omega$. To enhance the smoothness of $f$, we also wish to have some control on the distribution of $\mu_f$. To be more precise, we expect that if the distance between $z_1,z_2\in\mathcal{S}_1$ is small, the difference between their complex dilatations should also be small. This motivates us to add a new term to the energy provided that $\mu_f$ is regular enough to have a $L^2$ weak derivative. We can also minimize
\begin{equation}
\label{eq:energy_measure_qc}
    E(f) = \mathcal{H}(f_*\gamma|\nu) + \alpha\norm{\mu_f}_{L^2}^2 + \beta \norm{\nabla\mu_f}_{L^2}^2
\end{equation}
where $\beta>0$ is prescribed. Both energy functional shall be minimized over all quasi-conformal maps from $\mathcal{S}_1$ to $\mathcal{S}_2$ with bounded complex dilatation. The mapping theorem allows us to parameterize all these maps by their Beltrami coefficients, which are only assumed to be bounded and complex measurable.

\subsection{A Flow-based Solver}
\label{sec:algorithm}
The optimization of the energy~\eqref{eq:energy_measure_qc} seems difficult to handle. To see it, notice that the relative entropy is integrated on the target surface $\mathcal{S}_2$ and involves the pushforward measure and the target measure. However, the Beltrami coefficient is defined on the initial surface $\mathcal{S}_1$ and corresponds to the map $f$ through the Beltrami equation. Luckily, each term of the energy~\eqref{eq:energy_measure_qc} has a gradient flow in some appropriate spaces. Utilizing this crucial structure, we propose a splitting method to minimize the energy. The basic idea is that, each time, we update the map $f$ in the direction of the gradient flow of one term of the energy~\eqref{eq:energy_measure_qc}. To illustrate our method, we need to introduce several operators.

As introduced, the gradient flow of the relative entropy is the Fokker-Planck equation. Let $f:\mathcal{S}_1\to\mathcal{S}_2$ be a quasiconformal map. Suppose $f_*\gamma$ has density $\rho_0$ and the target measure $\nu$ has density $u=e^{-V}$ for some function $V$ that is $C^1$ on $\mathcal{S}_2$. Taking $\rho_0$ as the initial measure, the Fokker-Planck equation
\begin{equation}
\label{eq:Fokker_Planck_2}
    \partial_t\rho_t - \nabla\cdot(\nabla\rho_t+\rho_t\nabla V) = 0\ \ \ \text{in }\R^2\times [0,\infty)
\end{equation}
gives a time-dependent density $\rho_t$ that is supposed to converge to $\nu$. A rigorous analysis of the convergence of the Fokker-Planck equation can be found in~\cite{ambrosio2008gradient}. To find the transport map, we apply the continuity equation
\begin{equation}
\label{eq:continuity}
    \partial_t \rho_t + \nabla\cdot(v_t\rho_t) = 0
\end{equation}
to obtain a time-dependent velocity field $v_t$. Notice that $v_t$ is determined by both the initial density $\rho_0$ and the target density $u$. Suppose that the target density is fixed. We then define
\begin{equation}
\label{eq:transport_t}
    T_t(\rho_0)(x) = x + \int_0^t v_s\,ds \ \ \ \text{for any } x\in\R^2.
\end{equation}
Note that $[T_t(\rho_0)]_*\rho_0=\rho_t$ for any $t>0$. Thus, with a prescribed $t>0$, the map $f_t:\mathcal{S}_1\to\mathcal{S}_2$ defined by 
\begin{equation}
    f_t(x) = T_t(\rho_0)(f(x))\ \ \ \text{for }x\in\mathcal{S}_1
\end{equation}
is supposed to reduce the relative entropy, where $\rho_0$ is the density of $f_*\gamma$. We cannot promise that $f_t$ is quasiconformal with bounded complex dilatation. In case it is not, we introduce a projection operator in $L^{\infty}(\C)$ via
\begin{equation}
\label{eq:projection_mu}
    P_k (\mu)(z) =
    \begin{cases}
        \mu(z) & if \ \ \abs{\mu(z)}\leq k \\
        \dfrac{k}{|\mu(z)|}\,\mu(z) & if \ \ |\mu(z)|>k
    \end{cases}
\end{equation}
where $k\in [0,1)$ is prescribed. Suppose $f_t$ is regular enough to be associated with a Beltrami coefficient $\mu$. $\tilde{\mu}=P_k(\mu)$ will give us a modified Beltrami coefficient. We can then solve the mapping problem~\eqref{eq:qc_energy_1} or\eqref{eq:qc_energy_2} to give a quasiconformal modification $\tilde{f}_t=S_{\Gamma}(\tilde{\mu})$, where $\Gamma$ is a landmark condition or a Dirichlet condition. 

The gradient flow of the $L^2$-norm in $L^2(\C)$ is trivial. Suppose a quasiconformal map $f:\mathcal{S}_1\to\mathcal{S}_2$ has a Beltrami coefficient $\mu$. For a prescribed $t>0$, $e^{-t}\mu$ clearly reduces the $L^2$-norm. $S_{\Gamma}(e^{-t}\mu)$ then outputs the updated quasiconformal map.

The Laplace equation is widely known to be the gradient flow of the Dirichlet energy in $L^2(\R^n)$. Given an initial Beltrami coefficient $\mu_0=\sigma_0 + i\tau_0$, we denote by $\sigma_t$ and $\tau_t$ the solution of the Laplace equation
\begin{equation}
\label{eq:laplace_mu}
    \partial_s\sigma_s = \Delta\sigma_s,\ \ \ \partial_s\tau_s = \Delta\tau_s,
\end{equation}
with initial values $\sigma_0$ and $\tau_0$ respectively.  We then define the operator $R_t:L^2(\Omega,\C)\to L^2(\Omega,\C)$ to smooth Beltrai coefficients, which is given by 
\begin{equation}
\label{eq:smooth_mu}
    \mu_t:=R_t(\mu_0)=\sigma_t+i\tau_t
\end{equation}
for any $t>0$. The output $\mu_t$ is a smoothed Beltrami coefficient with reduced Dirichlet energy and gives an updated quasiconformal map $S_{\Gamma}(\mu_t)$. 

Our algorithm for minimizing the energy~\ref{eq:energy_measure_qc} is motivated by the splitting algorithms in optimization. Briefly speaking, we alternate between optimizing for one component of the energy~\ref{eq:energy_measure_qc} to obtain an updated map $f$ each time and then for another. A more comprehensive formulation can be described using the operators $S_{\Gamma}$~(\ref{subsec:qc2d}), $T_t$~(\ref{eq:transport_t}), $P_k$~(\ref{eq:projection_mu}), and $R_t$~(\ref{eq:smooth_mu}) we have already defined. We start from an initial orientation-preserving bijective map $f^{(0)}$, which can be realized as a quasiconformal map or a harmonic map. Suppose the output of the $(k-1)$-th iteration is $f^{(k-1)}$. In the $k$-th iteration, we first compute the density $\rho^k$ of $(f^{(k-1)})_*\gamma$. The evolution of the Fokker-Planck equation over time $t_1$ gives us an updated density function which is supposed to be closer to the density $u$ of $\nu$. We then solve the continuity equation to get the associated transport map $T_{t_1}(\rho^k)$. The composition $g^k = T_{t_1}(\rho^k)\circ f^{(k-1)}$ thus reduces the relative entropy. After that, we compute the Beltrami coefficient $\mu_{g^k}$ of the map $g^k$ and apply the projection~\eqref{eq:projection_mu} to make it strictly less than some $0<k<1$ by the projection operator $P_k$. Next, we turn to the optimization of Beltrami coefficients. The gradient flows of the energies $\norm{\mu}_{L^2}^2$ and $\norm{\nabla\mu}_{L^2}^2$ gives us an updated Beltrami coefficient $\mu^k:=R_{t_3}(e^{-t_2}\mu_{g^k})$. Solving the Beltrami equation with respect to some landmark condition or boundary condition $\Gamma$ gives an updated map $f^{(k)}:=S_{\Gamma}(\mu^k)$. This is the output of the iteration $k$. We stop the iteration when the result is satisfactory enough or the energy does not decay. For better performance, we may define the time steps as $t_{k,1}$, $t_{k,2}$, and $t_{k,3}$, which are parameters depending on the iteration number $k$.

The relative difference between parameters $t_1$, $t_2$, and $t_3$ depends on whether we emphasize on the difference between the pushforward measure and the target one or the quantity and smoothness of the Beltrami coefficients. However, the sensitivity to time steps of the numerical solvers for the entropy term and the quasiconformal term are different. Indeed, our experiments shows that the time step of the Fokker Planck equation must be set small enough to ensure convergence. This requires us to set a smaller time scale when we numerically discretize the operators $T_t$ and $R_t$. We are going to discuss it in the next section.

\section{Numerical methods}
\label{sec:numerical}
To numerically compute a measure-preserving map, we shall see that it suffices to deal with the case that $\mathcal{S}_1$ is discretized as a triangulated surface while $S_2$ is not. Indeed, when $\mathcal{S}_1$ and $\mathcal{S}_2$ are both triangulated surfaces, we can choose a template domain $\Omega\in\R^2$ such as the unit circle with standard Lebesgue measure, and find piecewise linear measure-preserving maps $f_1:\mathcal{S}_1\to\Omega$ and $f_2:\mathcal{S}_2\to\Omega$. The desired map from $\mathcal{S}_1$ and $\mathcal{S}_2$ can then be approximated by a suitable piecewise linear interpolation of $f_2^{-1}\circ f_1$. 

Now, suppose $\mathcal{S}_1=(V,E,F)$ be a triangulated surface and $\mathcal{S}_2$ be a continuous surface such as the unit circle, a rectangular domain, the Poincaré disk, or some domain in $\R^2$ with a prescribed metric. Recall that the uniformization theorem states that any simply-connected Riemann surface is conformally equivalent to the unit disk, the complex plane, or the Riemann sphere. Since this article does not touch on diffusion on curved surfaces, we may suppose that $\mathcal{S}_2$ is embedded in $\R^2$. Our numerical solver intends to approximates the minimizer of energy~\ref{eq:energy_measure_qc} find a piecewise linear quasiconformal map. For surfaces with boundaries, the boundaries are usually fixed by prescribing a Dirichlet boundary condition.
\subsection{Discretization of transport maps}
Given an arbitrary triangulated surface $\mathcal{S}=(V,E,F)$, we discretize measures $\gamma\in\mathcal{P}_2^{ac}(\mathcal{S})$ whose densities are positive everywhere by surface measures whose densities are strictly positive and piecewise constant on each triangle face. Let $(ijk)\in F$ represent a triangle face with vertices $i$, $j$, and $k$. The feasible set of the density functions on can thus be written as
\begin{equation}
    \{\rho \in \R_+^{|F|\times 1}: \sum_{(ijk)\in F} \rho(ijk)\cdot \text{Area}(ijk)=1\}
\end{equation}
where $\rho(ijk)$ is the density of the measure on triangle face $(ijk)$. Let $\rho$ be such a density function of a probability measure $\gamma$ on $\mathcal{S}$. Suppose $f:\mathcal{S}\to\R^3$ is a piecewise linear map that is homeomorphic onto its image. Clearly, the density $\rho_f\in \R_+^{|F|\times 1}$ of the pushforward measure $f_*\gamma$ is given by
\begin{equation*}
    \rho_f(f(ijk)) = \rho(ijk)\cdot\dfrac{\text{Area}(ijk)}{\text{Area}(f(ijk))}
\end{equation*}
where $f(ijk)$ is the image of the triangle $(ijk)$ under $f$. Notice that $f_*\gamma$ is a surface probability measure on $f(\mathcal{S})$ and has a piecewise constant density.

Our goal is to approximate the transport map $T_t$~\ref{eq:transport_t} from a piecewise constant measure $\gamma$ to a reference measure $\nu$ with density $e^{-V}$. It suffices to approximate the values of the vector fields $v_t$ at each vertex in $V$. Integrating the vector field $v_t$ then gives us a time-dependent piecewise linear map. This goal can be accomplished if we can discretize the Fokker-Planck equation and the continuity equation. But we shall see that the Fokker-Planck equation is a second order equation while $\gamma$ and $f_*\gamma$ do not have enough regularity to produce a second-order derivative as they have piecewise constant density. Nevertheless, what we desire is the transport map instead of the evolution of density. Comparing the Fokker-Planck equation~\eqref{eq:Fokker_Planck_2} and the continuity equation~\eqref{eq:continuity}, we get
\begin{equation}
    v_t = -\dfrac{\nabla\rho_t}{\rho_t} - \nabla V
\end{equation}
Suppose that we already have an explicit formula for the drift $V$. Then, once we have an approximation of $\rho_t$ and $\nabla\rho_t$ at each vertex, we get an approximation of $v_t$, hence the transport map. Inspired by the work~\cite{bessemoulin2012finite}, we apply a finite-volume method to achieve it. Suppose $\rho$ is the density of $f_*\gamma$ for some $f:\mathcal{S}_1\to\R^2$. For an arbitrary vertex $i\in V$, we write $i\sim (ijk)$ if $i$ is a vertex of face (ijk). Let $\rho_i$ denote the approximated density at vertex $i$. We approximate it by averaging over all neighbouring faces of $i$. That is,
\begin{equation}
    \rho_i = \sum_{i\sim (ijk)}\text{Area}(ijk)\cdot\rho(ijk) \bigg/ \sum_{i\sim (ijk)}\text{Area}(ijk)
    \label{eq:density_vertex}
\end{equation}
With approximated $\rho_i$ at each vertex $i$, we extend it to a piecewise linear function $\tilde{\rho}$ on $f(\mathcal{S}_1)$. The gradient of $\tilde{\rho}$ is piecewise constant on each face and we denote the value on face $(ijk)$ by $\nabla\tilde{\rho}(ijk)$. We then approximate the gradient of $\rho$ at vertex $i$ by
\begin{equation}
    (\nabla\rho)_i = \sum_{i\sim (ijk)}\text{Area}(ijk)\cdot\nabla\tilde{\rho}(ijk) \bigg/ \sum_{i\sim (ijk)}\text{Area}(ijk)
    \label{eq:nabla_density_vertex}
\end{equation}
With the approximation of the vector, we give an explicit Euler scheme to discretize the transport map $T_t$ defined by~\eqref{eq:transport_t}. Suppose $\rho_0$ is the density of $(f_0)_*\gamma$. Given $t>0$, we fix a smaller time scale $\tau$ such that $t=n\tau$. For each vertex $w_i$ and $0\leq k \leq n-1$, we define
\begin{equation}
    f_{(k+1)\tau}(w_i) = f_{k\tau}(w_i) + \tau\cdot v_{k\tau}^i
\end{equation}
where $v_{k\tau}^i$ is the value at vertex $w_i$ of the velocity field $v_{k\tau}$ associated with $(f_{k\tau})_*\gamma$. The discrete transport map $T_t(\rho_0)$ is characterized by
\begin{equation}
    [T_t(\rho_0)\circ f_0] (w_i) = f_{t}(w_i)
\end{equation}
for any appropriate map $f_0$ and vertex $w_i\in V$. To evaluate the performance of a transport map, we need to numerically approximate the relative entropy~\eqref{eq:relative_entropy}. Let $\rho$ be the density of $f_*\gamma$. Since $\rho$ is piecewise constant and $\nu$ has density $e^{-V}$, we have
\begin{align}
    \mathcal{H}(f_*\gamma|\nu) &= \int_{f(\mathcal{S}_1)}\rho\log(\rho) + \rho V \, dx \\
    &= \sum_{(ijk)\in F}\int_{f(ijk)} \rho(ijk)\log(\rho(ijk)) + \rho(ijk)V(x)\, dx \\
    &= \sum_{(ijk)\in F} \text{Area}(f(ijk))\cdot\rho(ijk)\log(\rho(ijk)) \\
    & \ \ \ \ \ + \sum_{(ijk)\in F}\rho(ijk)\int_{f(ijk)}V(x)\,dx
\end{align}
To numerically compute this integral, we only need to approximate $V(x)|_{(ijk)}$ for each triangle face $(ijk)$. If the area of the triangle faces are smaller enough and $V$ is continuous, we can simply approximate $V|_{(ijk)}$ by a constant for each $(ijk)$. Another possible choice is to interpolate $V$ by a piecewise linear function with the values $V(w_i)$ for all vertex $w_i$. Suppose we approximate $V$ by $\tilde{V}$ to compute the approximated entropy $\tilde{\mathcal{H}}$. The difference between $\tilde{\mathcal{H}}$ and the ground truth $\mathcal{H}$ satisfies
\begin{equation}
    |\mathcal{H}-\tilde{\mathcal{H}}| \leq \sum_{(ijk)\in F}\rho_{ijk}\int_{f(ijk)}|V(x)-\tilde{V}(x)|\,dx
\end{equation}
which is clearly controllable if the mesh has small size and is regular enough.

\subsection{Discretization of quasi-conformal maps}
As mentioned, we approximate the homeomorphism from a triangulated surface to a target surface by an orientation-preserving piecewise linear map. Since the differentials $f_z$ and $f_{\overline{z}}$ of a piecewise linear map are piecewise constant, the Beltrami coefficient $\mu=f_z/f_{\overline{z}}$ of it is also piecewise constant on each triangle face. Given an orientation-preserving piecewise linear map, it is straightforward to compute its Beltrami coefficient. But what is more subtle is the converse problem. Given a triangulated surface $\mathcal{S}\subset\C$ and a piecewise constant Beltrami coefficient $\mu$, the mapping theorem~\ref{thm:measurable_riemann} guarantees the existence of a quasi-conformal map that satisfies the Beltrami equation given by $\mu$. However, this map is not guaranteed to be piecewise linear. Even if we only fix two points, it is often impossible to find a piecewise linear map whose Beltrami coefficient equals the given one almost everywhere. But still, we can find a minimizer of the energies \ref{eq:qc_energy_1} or \ref{eq:qc_energy_2} over all piecewise linear maps under some landmark or boundary condition. We call the minimizer a least-square quasiconformal map. So, the key point of our numerical method is to give a discretized form of energy~\eqref{eq:qc_energy_1} and~\eqref{eq:qc_energy_2}. We are going to illustrate it in detail.

Let $\mathcal{S}=(V,E,F)$ be a triangulated surface and $f$ be a piecewise linear map. We enumerate all the vertices of $\mathcal{S}$ as $V:=\{w_1,w_2,\cdots,w_n\}$ and denote their images under $f$ by $\{f(w_1),f(w_2),\cdots,f(w_n)\}=\{u_1+iv_1,u_2+iv_2,\cdots,u_n+iv_n\}$, where all of $u_i$ and $v_i$ are real. Let vectors $u=(u_1\ u_2 \cdots u_n)^t$ and $v=(v_1\ v_2 \cdots v_n)^t$ denote the $x$ and $y$ coordinates of the image of the vertices under $f$. We then discretize the Dirichlet type energies $E_A(u)=\frac{1}{2}\int_{\mathcal{S}}||A^{1/2}\nabla u||^2$ and $E_A(v)=\frac{1}{2}\int_{\mathcal{S}}||A^{1/2}\nabla v||^2$ appearing in~\eqref{eq:qc_energy_1}. Let $T\in F$ be an arbitrary triangle face with vertices $w_0^T,w_1^T,w_2^T \in \R^{2\times 1}$. Suppose the image of $T$ under $f$ is the triangle with vertices $\{f(w_0^T),f(w_1^T),f(w_2^T)\}=\{u_0^T+iv_0^T,u_1^T+iv_1^T,u_2^T+iv_2^T\}$. Since $f$ is linear on $T$, we can express the gradient of $f$ on $T$ as linear functions of $\{u_0^T,u_1^T,u_2^T\}$ and $\{v_0^T,v_1^T,v_2^T\}$
\begin{equation}
    \nabla u|_T = \dfrac{1}{2\text{Area}(T)}\begin{pmatrix}0&-1\\1&0\end{pmatrix}\sum_{i=0}^2u_i(w_{2+i}^{T}-w_{1+i}^{T}),
\end{equation}
and
\begin{equation}
    \nabla v|_T = \dfrac{1}{2\text{Area}(T)}\begin{pmatrix}0&-1\\1&0\end{pmatrix}\sum_{i=0}^2v_i(w_{2+i}^{T}-w_{1+i}^{T}).
\end{equation}
where the sub-indices of $w^T_{2+i}$ and $w^T_{1+i}$ are modulo $3$. Since the Beltrami coefficient $\mu$ is constant on $T$, the symmetric dilation matrix $A$ define by~\eqref{eqt:dilation_matrix} is constant on $T$. By canceling out $\text{Area}(T)$, the integrals $\int_T ||A^{1/2}\nabla u||^2$ and $\int_T ||A^{1/2}\nabla v||^2$ are quadratic functions of $\{u_0^T,u_1^T,u_2^T\}$ and $\{v_0^T,v_1^T,v_2^T\}$ respectively. By summing over all faces $T\in F$, we obtain two quadratic forms
\begin{equation}
    E_A(u) = u^t L_{\mu} u \ \ \text{and}\ \ E_A(v) = v^t L_{\mu} v,
\end{equation}
where the symmetric matrix $L_{\mu}$ depending on $\mu$ is referred as the \emph{generalized Laplacian matrix}. When $\mu$ is identically $0$, $L_{\mu}$ coincides with the standard Laplacian matrix on triangulated surfaces~\cite{pinkall1993computing}. We also need to discretize the area functional $\mathcal{A}(u,v)$. Since $f$ is an orientation-preserving homeomorphism,
\begin{equation}
    \mathcal{A}(u,v) = \int_{\mathcal{S}} u_x v_y - v_x u_y = \begin{pmatrix} u^t & v^t\end{pmatrix} \begin{pmatrix} 0 & U \\ -U & 0 \end{pmatrix} \begin{pmatrix} u \\ v \end{pmatrix}
    \label{eqt:discretized_area}
\end{equation}
where $U$ is a skew-symmetric matrix. Define the symmetric matrix
\begin{equation}
    M = \begin{pmatrix} \mathcal{L}_{\mu} & 0 \\ 0 & \mathcal{L}_{\mu} \end{pmatrix} - \begin{pmatrix} 0 & U \\ -U & 0 \end{pmatrix}.
    \label{eqt:M}
\end{equation}
Thus, the energy~\eqref{eq:qc_energy_1} can be discretized as
\begin{equation}
     E_{\text{QC}}^{\mu}(u,v) = \begin{pmatrix} u^t & v^t\end{pmatrix} M \begin{pmatrix} u \\ v \end{pmatrix}.
\end{equation}
Ww remark that when the Beltrami coefficient $\mu$ is identically zero, $E_{\text{QC}}$ coincides with the energy appearing in the least-square conformal map~\cite{desbrun2002intrinsic}. We cannot get a quasiconformal map by directly optimizing this energy because it has infinitely many minimizers. To get a valid solution, we need to fix some points. Indeed, fixing two points is enough to guaranty a unique solution. This can be seen from the following theorem.

\begin{theorem}
    \label{thm:numerical_qc}
    Suppose the Beltrami coefficient $\mu$ is strictly less than $1$, and the triangle mesh $\mathcal{S}$ is connected and has no dangling points (i.e. there are no triangles that share a common vertex but no common edge). Let $I_{pin}$ be the indices of points to be pinned with cardinality $|I_{pin}|\geq 2$, $I_{free}$ be the indices of the free points, and M be the matrix defined by the formula~\ref{eqt:M}. Then, the $2|I_{free}|\times2|I_{free}|$ sub-matrix $M_{free}$ of $M$ indexed by the free points has full rank.
\end{theorem}
\begin{proof}
    The conformal version of this theorem is stated and proved in~\cite{levy2023least}. Following the main ideas, the work~\cite{qiu2019computing} generalize the theorem to the quasiconformal case and give a similar proof. Here, we only sketch the basic ideas. The key observation is that a triangle mesh satisfying our assumption can be constructed incrementally using two operations. The first is gluing, which adds a new vertex and connect it to two neighboring vertices. The second is joining, which joins two existing vertices. One then proceeds by proving the incremental construction preserves the full rank property. Following the arguments in~\cite{levy2023least}, we eventually get a full rank matrix, which is identical to $M$.
\end{proof}
No matter whether we fix $2$ points or prescribe a Dirichlet boundary condition, this theorem guaranties that there is a unique least-squares quasiconformal map. To preserve the numerical stability, we do not arbitrarily choose points when we need to pin two. Instead, we choose two points that are far from each other in $\mathcal{S}$ and map them to the origin and $(1,0)$. We make a final remark that the input of the quasiconformal solver must be a triangle surface embedded in $\R^2$. When we have a surface embedded in $\R^3$, we need to map it to $\R^2$ first using the free boundary conformal map~\cite{desbrun2002intrinsic}.

We now discuss the smoothing operators introduced in the previous section to minimize the energy~\eqref{eq:energy_measure_qc}. The gradient flow of the functional $\norm{\mu}_{L^2}^2$ has an explicit formula. So, if we fix $t>0$ and let $\mu\in\C^{|F|\times 1}$ denote a piecewise constant Beltrami coefficient on $\mathcal{S}$, then $e^{-t}\mu$ will give us an updated Beltrami coefficient with reduced $L^2$ norm. However, we need to be more careful when discretizing the gradient flow of $\norm{\nabla\mu}_{L^2}^2$ because the regularity of $\mu$ is only piecewise constant on each face. To tackle this problem, we consider utilizing the discrete Laplace-Beltrami operator defined on triangulated surfaces. But another problem arises that the Beltrami coefficient $\mu$ is defined on each face instead of each vertex. This motivates us to consider the dual triangulation of $(V,E,F)$ whose vertices are exactly the faces of the primal one. Suppose $g:F\to\R$ is an arbitrary function defined on all faces. For an arbitrary face $f\in F$, the dual Laplacian $\Delta^*$ of $g$ at $f$ is given by
\begin{equation}
    \Delta^* g(f) = \dfrac{1}{\text{Area}(f)}\sum_{e\in\partial f}\omega_e(g(f')-g(f))
\end{equation}
where $f'$ is the neighbouring face across the edge $e$, and positive numbers $\omega_e$ are the edge weights that can be defined by the famous cotangent weight~\cite{pinkall1993computing}. With the dual Laplacian, we discretize the smoothing operator characterized by~\eqref{eq:laplace_mu} by an explicit Euler scheme. First we choose a smaller time scale $\tau$ such that $t=n\tau$. Given an arbitrary Beltrami coefficient $\mu:F\to\C$, we take it as $\mu_0$. For each $0\leq k\leq n-1$, we define
\begin{equation}
    \mu_{(k+1)\tau} = \mu_{k\tau} + \tau\Delta^*\mu_{k\tau}
\end{equation}
where $\Delta^*$ acts on the real part and imaginary part of $\mu_{k\tau}$ separately. We thus define $R_t(\mu)=\mu_{n\tau}$, which gives a smoothed Beltrami coefficient.

\subsection{The initial conformal structure}
In section~\ref{sec:formulation}, we mention that we need to parameterize a surface $\mathcal{S}\in\R^3$ to $\R^2$ first in order to make the Beltrami coefficients well-defined. One should notice that this parameterization need not be the initial map $f_0$ for minimizing the energy~\ref{eq:energy_measure_qc}. Also, our numerical algorithm for quasiconformal maps works only for triangulated surfaces in $\R^2$. One approach to parameterizing a surface is to compute the least-square conformal map by pinning two points. However, this method may not output a good parameterization under the following two cases. The first is that the surface $\mathcal{S}$ is highly curved. That is, the discrete curvature at some part of the surface is very large. The second case is that the mesh structure of the triangulated surface $\mathcal{S}$ is not good enough. As a result, we need to propose some methods to ensure a parameterization with a good mesh quality, which we also call a good initial conformal structure. 

\begin{figure}
\begin{subfigure}[t]{.42\textwidth}
\centering
\includegraphics[width=.9\linewidth]{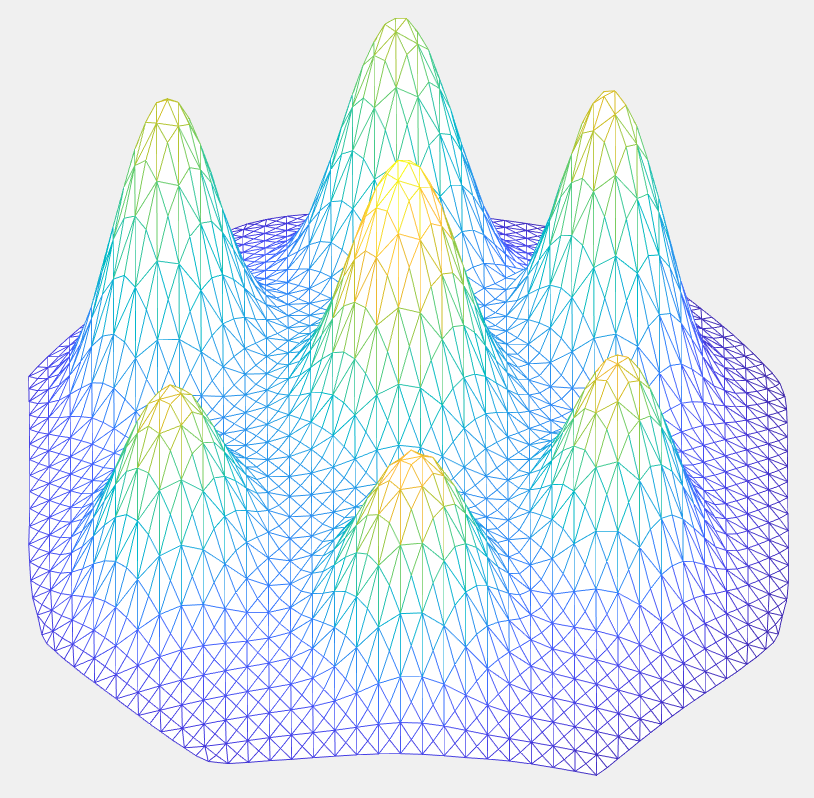}
\caption{A highly curved surface}
\label{fig:peaks_original}
\end{subfigure}%
\hfill
\begin{subfigure}[t]{.42\textwidth}
\centering
\includegraphics[width=.9\linewidth]{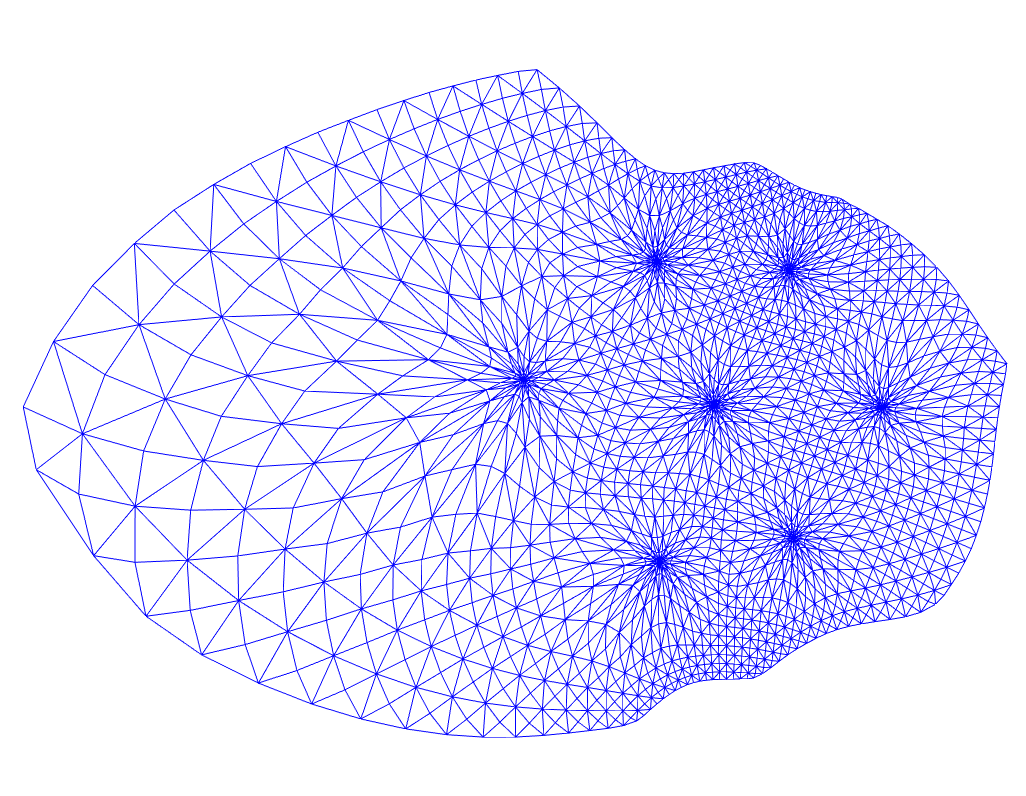}
\caption{The initial parameterization}
\label{fig:peaks_dncp}
\end{subfigure}%
%\caption{A genus-3 example}
\\
\begin{subfigure}[t]{.42\textwidth}
\centering
\includegraphics[width=.9\linewidth]{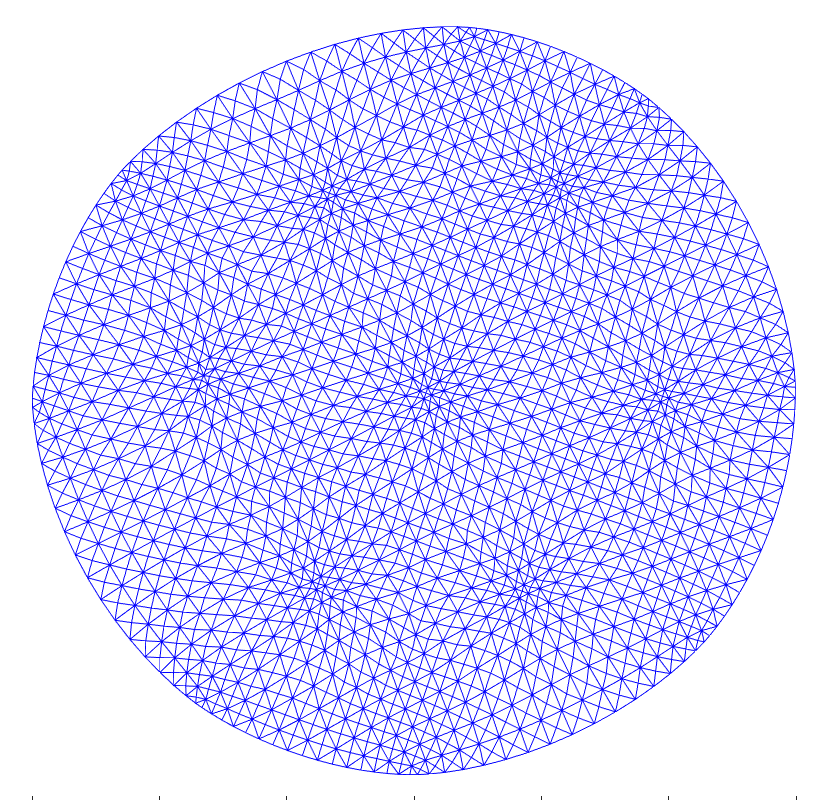}
\caption{The one given by a flat metric}
\label{fig:peaks_ricci}
\end{subfigure}%
\hfill
\begin{subfigure}[t]{.42\textwidth}
\centering
\includegraphics[width=.9\linewidth]{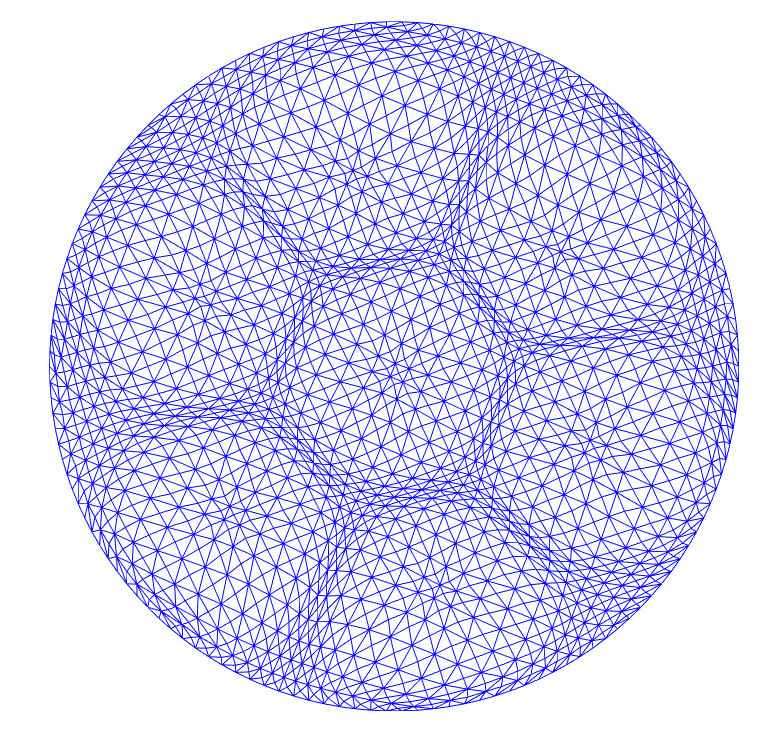}
\caption{The final result}
\label{fig:peaks_final}
\end{subfigure}%
%\caption{A genus-3 example}
\end{figure}

One method is to apply the discrete Ricci flow algorithm~\cite{jin2008discrete} to output a flat discrete metric on $\mathcal{S}$. A discrete metric is a function $l:E\to\R^+$ from the edges of $\mathcal{S}$ to positive real numbers. A discrete metric is called flat if the discrete Gaussian curvature is zero everywhere. By outputting a flat metric, we can remove the negative effects caused by the large curvature. For example, Fig.~\ref{fig:peaks_original} shows a triangulated surface with intense curvature. Fig.~\ref{fig:peaks_dncp} shows the parameterization given by the free boundary conformal map, whose quality is not good enough for us to solve the optimization problem. Fig.~\ref{fig:peaks_ricci} shows the parameterization given by a flat metric. Taking use of it, we compute an area-preserving parameterization, which is shown in Fig.~\ref{fig:peaks_final}. 

\begin{figure}
\begin{subfigure}[t]{.45\textwidth}
\centering
\includegraphics[width=.9\linewidth]{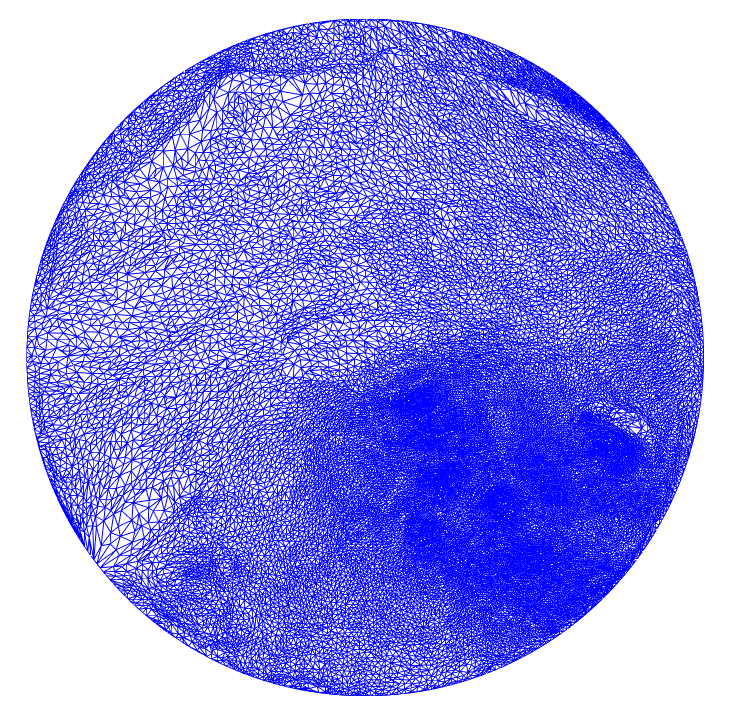}
\caption{The original mesh structure}
\label{fig:original_mesh}
\end{subfigure}%
\hfill
\begin{subfigure}[t]{.45\textwidth}
\centering
\includegraphics[width=.9\linewidth]{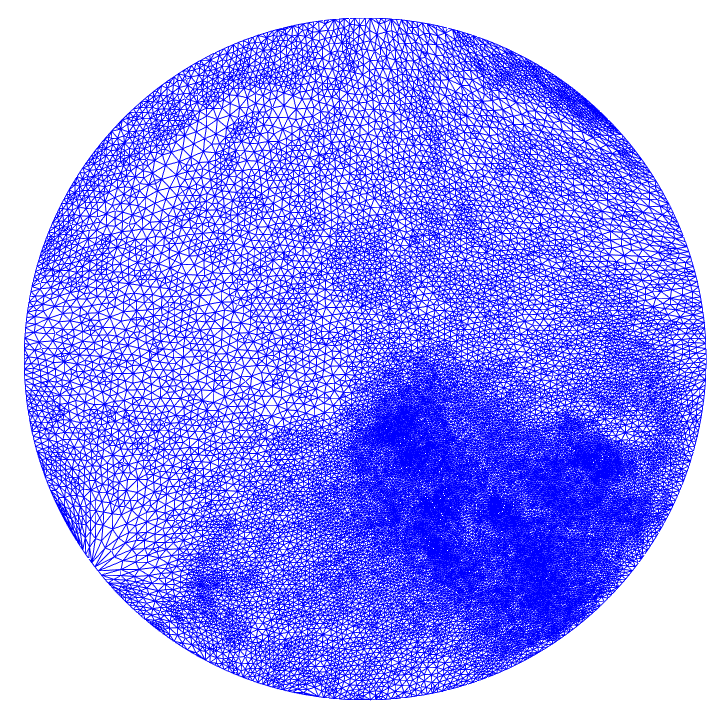}
\caption{The updated mesh structure}
\label{fig:updated_mesh}
\end{subfigure}%
\end{figure}

Another method is to apply the quasiconformal theory. Suppose we have a triangle mesh $(V,E,F)$ embedded in $\R^2$. We hope to comppute a Beltrami coefficient $\mu$ such that our Beltrami solver with this $\mu$ can output a mesh with better quality. Usually, a mesh structure is considered to be good if most triangles are close to an equilateral triangle. For each face $(ijk)\in F$, we can compute a linear map such that this face is mapped to an equilateral triangle by this map. The gradient of this map gives us a unique Beltrami coefficient $\mu$ defined on $(ijk)$. Completing this computation for each triangle face gives us a Beltrami coefficient that is piecewise constant on each face. Solving the optimization problem~\ref{eq:qc_energy_1} with respect to this $\mu$ gives us a least-square quasiconformal map such that most triangles faces are closer to equilateral triangles. Figure~\ref{fig:original_mesh} and Figure~\ref{fig:updated_mesh} demonstrate the effectiveness of this method. while the left shows a triangle mesh with a bad mesh structure, the right shows an updated one with much better quality.

\subsection{The genus-1 case}
The numerical algorithm for genus-0 surfaces is clearly presented by previous sections. In addition, our algorithm also works for closed surfaces of genus-1 under some special processing of them. As a direct consequence of the uniformization theorem, the universal cover of a genus-1 closed Riemann surface is the complex plane. In particular, their fundamental group is isomorphic to the free abelian group $G$ generated by two vectors in $\R^2$ that are not parallel to each other. So, the target surface can be chosen as a flat torus $\mathcal{T}$ that can be expressed as $\R^2/G$, such as the flat torus $\R^2/\mathbb{Z}^2$.

\begin{figure}
\begin{subfigure}[t]{.46\textwidth}
\centering
\includegraphics[width=.9\linewidth]{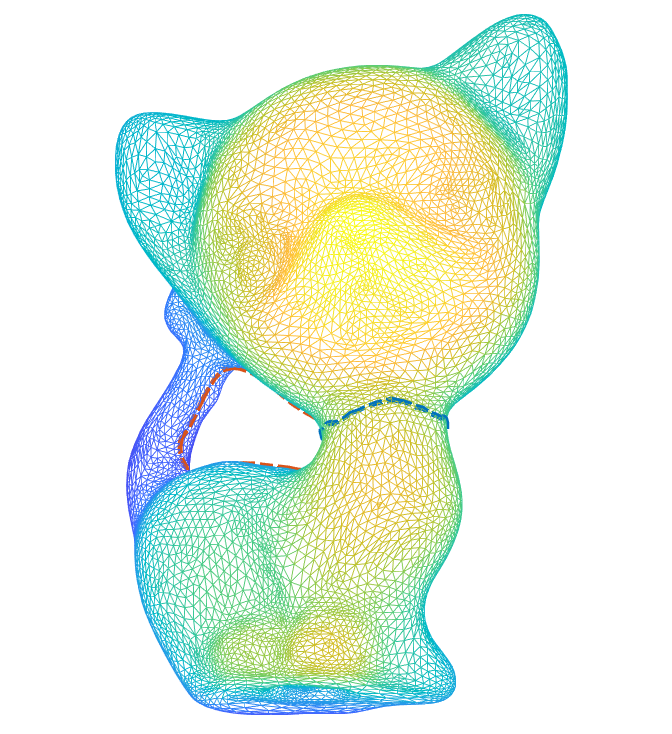}
\caption{A genus-1 surface with cuts}
\label{fig:g1_original}
\end{subfigure}%
\begin{subfigure}[t]{.54\textwidth}
\centering
\includegraphics[width=.9\linewidth]{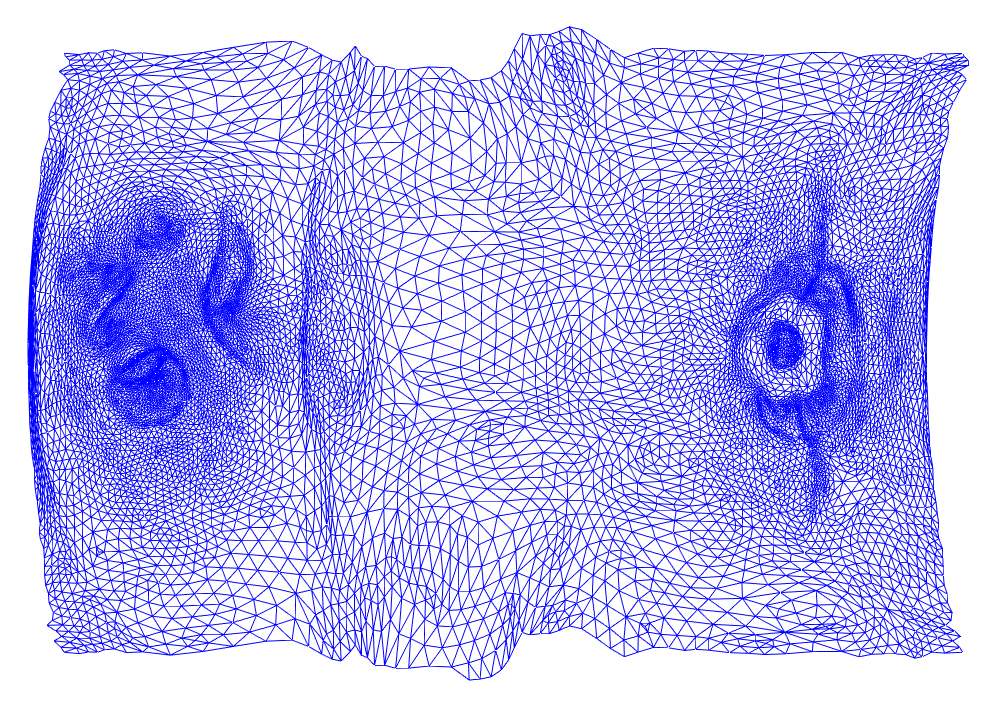}
\caption{The final map}
\label{fig:g1_final}
\end{subfigure}%
%\caption{A genus-3 example}
\end{figure}

However, given a triangulated torus $\mathcal{S}$ represented by $(V,E,F)$, we cannot directly embed in $\R^2$. To achieve it, we need to cut it open along a homology basis. We denote the open surface $\tilde{\mathcal{S}}$ obtained by $(\tilde{V},\tilde{E},F)$. Compared to $(V,E,F)$, $(\tilde{V},\tilde{E},F)$ has some duplicate vertices and edges. Let $\pi_1:\tilde{V}\to V$ denote the projection from $\tilde{\mathcal{S}}$ to $\mathcal{S}$ and $\pi_2$ denote the projection from the fundamental domain in $\R^2$ to the torus $\mathcal{T}$. The map $f:\mathcal{S}\to\mathcal{T}$ can then be represented by a map $\tilde{f}$ from $\tilde{\mathcal{S}}$ to a rectangular domain, under the constraint that 
\begin{equation}
    \pi_2(\tilde{f}(v_1)) = \pi_2(\tilde{f}(v_2))\ \ \text{if  }\pi_1(v_1) = \pi_1(v_2)
\end{equation}
The initial map $\tilde{f}_0$ can be found by a harmonic map with appropriate Dirichlet boundary condition. To find the desired evolution $\tilde{f}_t$, we need to take extra care. The reason is that a boundary vertex $i$ of $(\tilde{V},\tilde{E},F)$ corresponds to an interior vertex in $(V,E,F)$. So, the boundary vertices of $(V,E,F)$ should be allowed to move freely. Besides, the number of faces connected to $i$ in $(V,E,F)$ is less than the number of faces connected to it in the universal cover. So, when computing the gradient~\ref{eq:density_vertex} and the density gradient~\ref{eq:nabla_density_vertex} for this $i$, we need to count contributions from all faces connected to it in the universal cover. Fig.~\ref{fig:g1_original} shows a genus-1 surface together with a homology basis of it. We cut the surface along the basis and map it to a flat torus. Fig.~\ref{fig:g1_final} shows the final result.

\section{Experiments}
\subsection{The necessity of the quasiconformal term}
One may ask whether the quasiconformal term in~\ref{eq:energy_measure_qc} should be kept when we only need to compute a measure-preserving map. Actually, as we will see in the following example, the quasiconformal term takes on the effect of regularizing the mesh quality when the initialization map is not satisfactory enough. In the example, we set the density of the initial measure as a constant and the density of the target measure of $e^{-2|x|^2+c}$ for some $c$. Fig.~\ref{fig:e1_original} shows the initial mesh. Fig.~\ref{fig:e1_initial} shows a bad initialization map that we intentionally choose for the experiment.

\begin{figure}
\begin{subfigure}[t]{.42\textwidth}
\centering
\includegraphics[width=.9\linewidth]{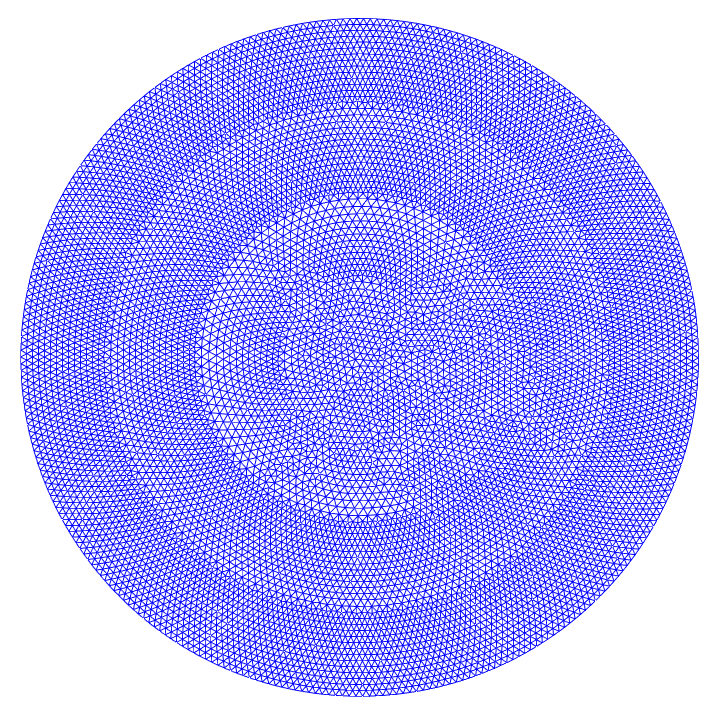}
\caption{The original mesh}
\label{fig:e1_original}
\end{subfigure}%
\hfill
\begin{subfigure}[t]{.42\textwidth}
\centering
\includegraphics[width=.9\linewidth]{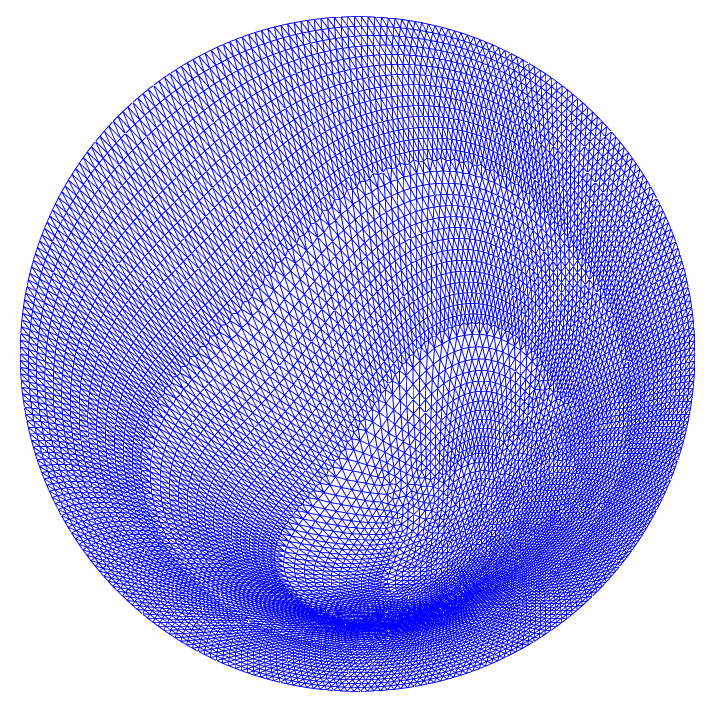}
\caption{The initial map}
\label{fig:e1_initial}
\end{subfigure}%
%\caption{A genus-3 example}
\\
\begin{subfigure}[t]{.42\textwidth}
\centering
\includegraphics[width=.9\linewidth]{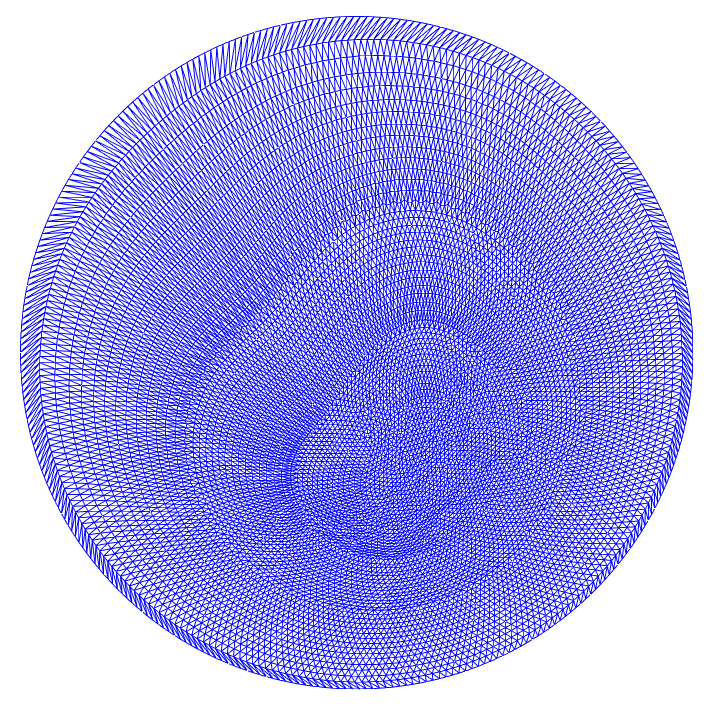}
\caption{Without quasiconformal regularization}
\label{fig:e1_result1}
\end{subfigure}%
\hfill
\begin{subfigure}[t]{.42\textwidth}
\centering
\includegraphics[width=.9\linewidth]{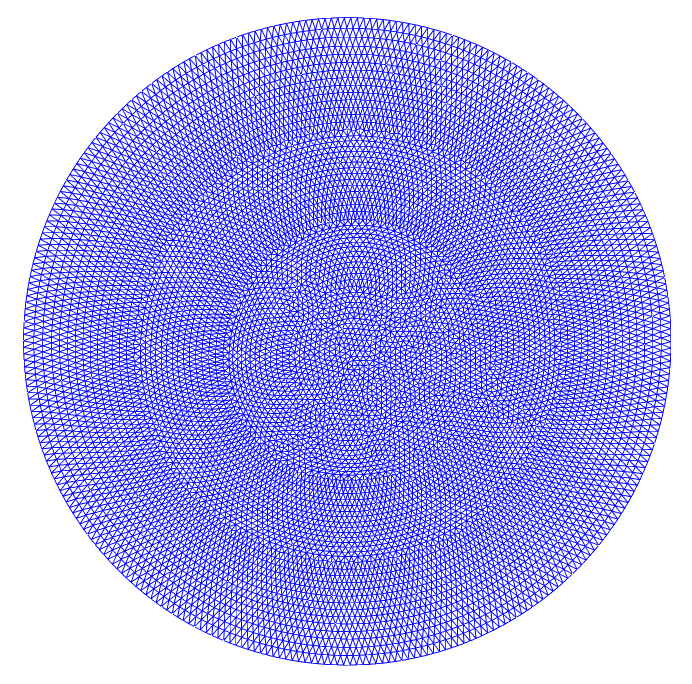}
\caption{With quasiconformal regularization}
\label{fig:e1_result2}
\end{subfigure}%
%\caption{A genus-3 example}
\end{figure}

Fig.~\ref{fig:e1_result1} shows the minimizer of~\ref{eq:energy_measure_qc} without any quasiconformal regularization. As we can see, this mesh is not regular enough, especially near the boundary. Fig.~\ref{fig:e1_result2} shows the minimizer of~\ref{eq:energy_measure_qc} with quasiconformal regularization. Even though the quality of the initial map is not good enough, it can be highly intensified through the optimization process.

\subsection{The effects of parameters $t_1$, $t_2$, and $t_3$}
In the presentation of our algorithm~\ref{sec:algorithm}, we propose to alternate between optimizing for one component of the energy~\ref{eq:energy_measure_qc} to have an updated mapping. Whether we attach more importance to the relative entropy or the quasiconformality depends on the choice of the parameters $t_1$, $t_2$, and $t_3$. If we hope to obtain a final result with smaller relative entropy, we just make $t_1$ larger, and vice versa. The effects of the parameters are shown intuitively in the following experiment.

\begin{figure}
\begin{subfigure}[t]{.33\textwidth}
\centering
\includegraphics[width=.9\linewidth]{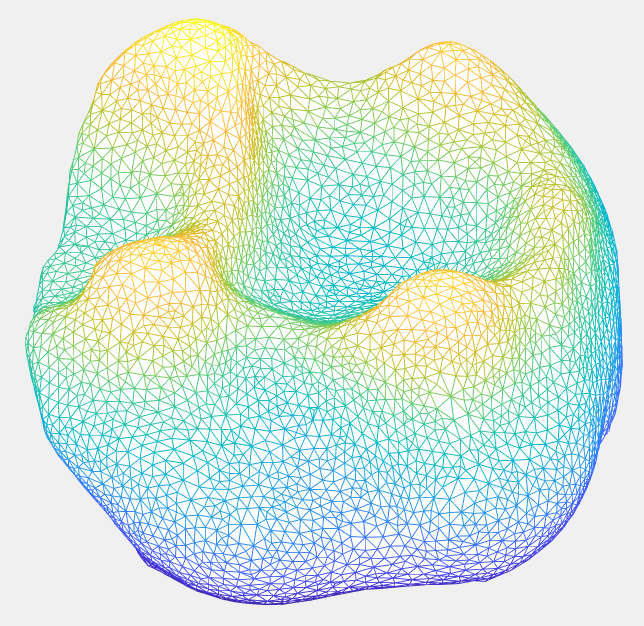}
\caption{A surface modeling a tooth}
\label{fig:e2_original}
\end{subfigure}%
\begin{subfigure}[t]{.33\textwidth}
\centering
\includegraphics[width=.9\linewidth]{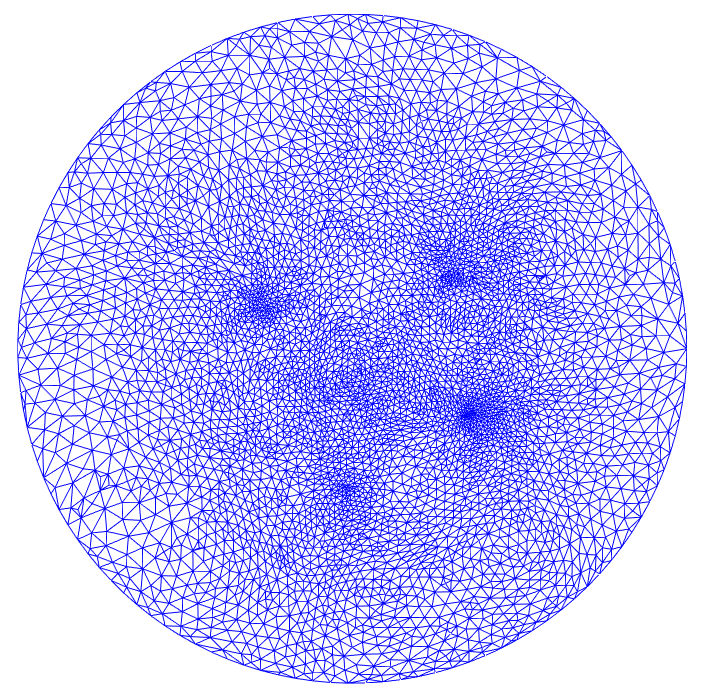}
\caption{The initial map}
\label{fig:e2_initial}
\end{subfigure}%
%\caption{A genus-3 example}
\begin{subfigure}[t]{.33\textwidth}
\centering
\includegraphics[width=.9\linewidth]{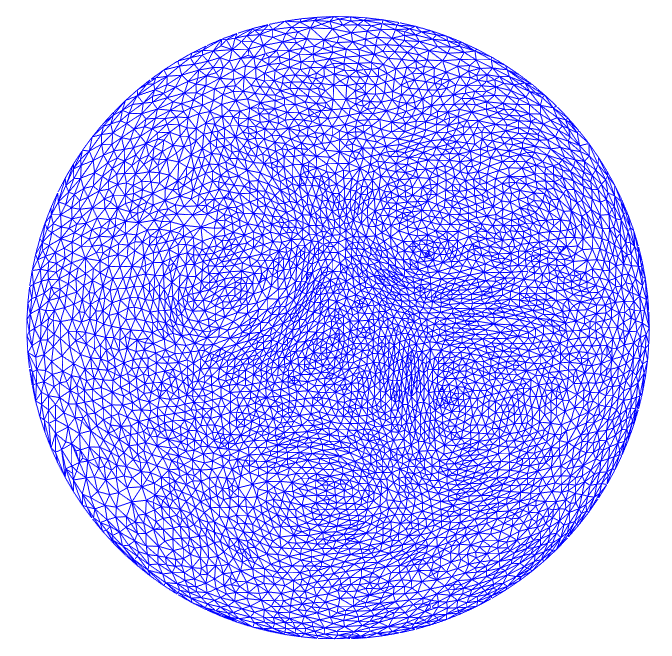}
\caption{Result 1}
\label{fig:e2_result1}
\end{subfigure}
\begin{subfigure}[t]{.33\textwidth}
\centering
\includegraphics[width=.9\linewidth]{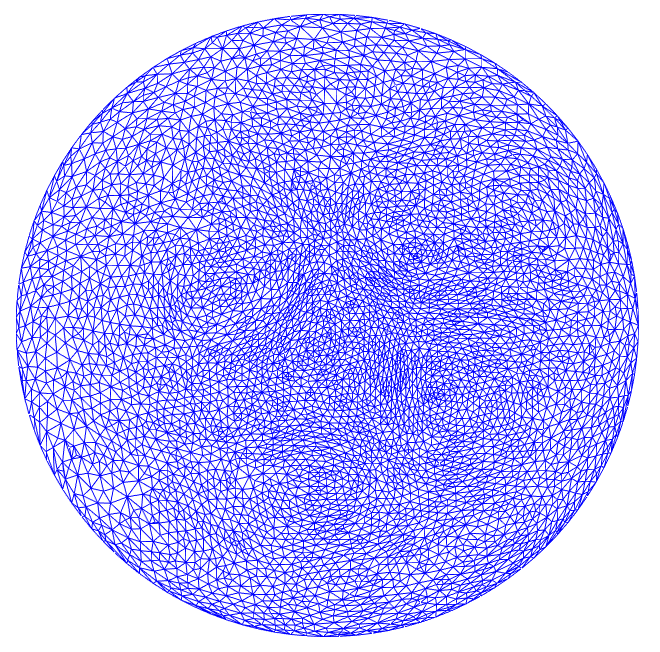}
\caption{Result 2}
\label{fig:e2_result2}
\end{subfigure}
\begin{subfigure}[t]{.32\textwidth}
\centering
\includegraphics[width=.9\linewidth]{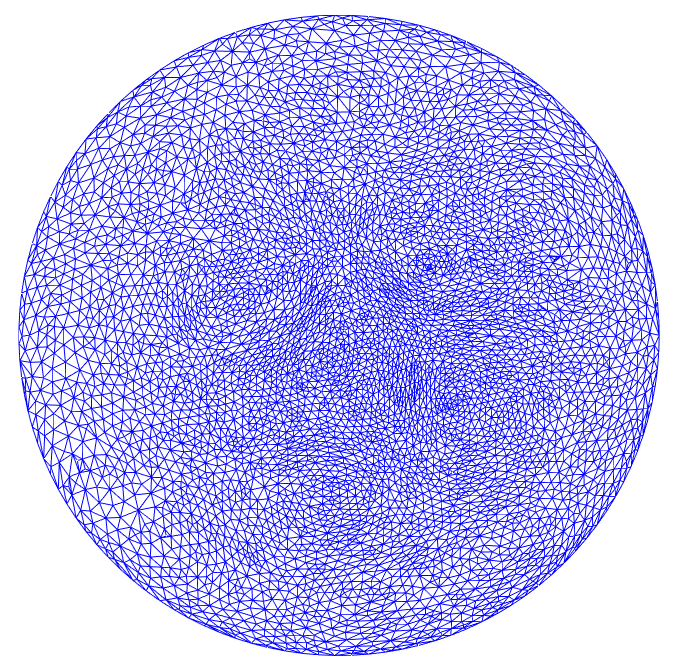}
\caption{Result 3}
\label{fig:e2_result3}
\end{subfigure}
\begin{subfigure}[t]{.32\textwidth}
\centering
\includegraphics[width=.9\linewidth]{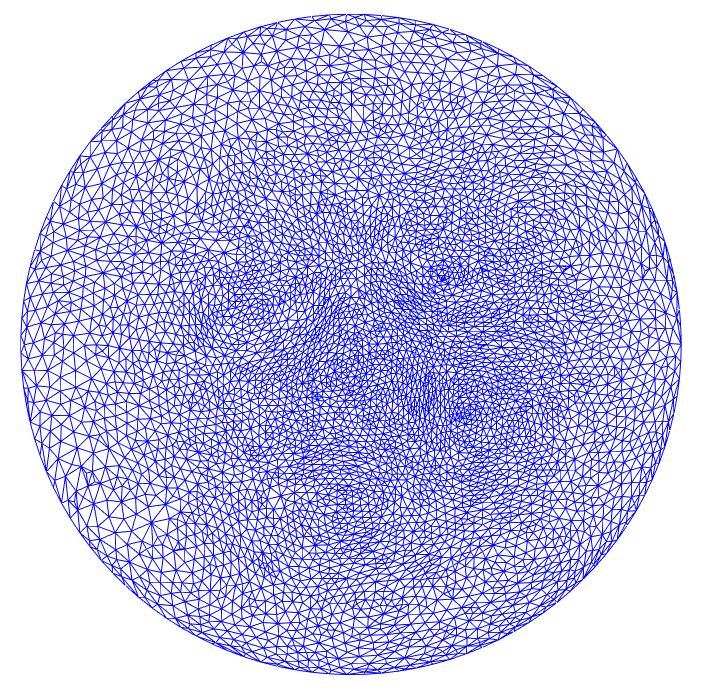}
\caption{Result 4}
\label{fig:e2_result4}
\end{subfigure}

\end{figure}

Fig.~\ref{fig:e2_original} shows a surface that models a tooth. Fig.~\ref{fig:e2_initial} shows the initial harmonic map with respect to a Dirichlet boundary condition. The probability measure on the surface is induced by the area and the reference probability measure is the Lebesgue measure divided by a constant. The initial conformal structure is  given by a free boundary conformal parameterization of the surface. Fig.~\ref{fig:e1_result1}, Fig.~\ref{fig:e2_result2}, Fig.~\ref{fig:e2_result3}, and Fig.~\ref{fig:e2_result4} demonstrate the results of the optimization problem with respect to 4 sets of parameters. The relative entropy of the initial map is 0.2695 and the $L^2$ norm of the Beltrami coefficient of it is 0.0276. The relative entropy and the $L^2$ norm of Beltrami coefficients of the results, together with the parameters, are shown in the table~\ref{tab:Ex2}. For each result, we run 50 iterations with the parameters chosen. Notice that when $t_2$ and $t_3$ become larger, the final entropy will increase and the final $||\mu||_{L^2}$ will decrease. We make a final remark that $t_1$, and $t_3$ are not the levels of discretization discussed in~\ref{sec:algorithm}. For better convergence, a smaller time scale is chosen for the discretization.

\begin{table}[t!]
    \centering
    \begin{tabular}{c|c|c|c|c|c}
         & $t_1$ & $t_2$ & $t_3$ & Final entropy & final $||\mu||_{L^2}$ \\ \hline 
        Result 1 & 0.0015 & 0 & 0 & 0.0314 & 0.2241 \\
        Result 2 & 0.0015 & 0.01 & 0.01 & 0.0498 & 0.1946 \\
        Result 3 & 0.0015 & 0.02 & 0.03 & 0.0686 & 0.1732 \\
        Result 4 & 0.0015 & 0.04 & 0.05 & 0.0997 & 0.1458
    \end{tabular}
    \caption{Some set of parameters and corresponding errors}
    \label{tab:Ex2}
\end{table}

\subsection{Some remeshing results}
One advantage of the algorithm proposed in this article is that one can achieve various data distributions on a mesh by specifying different initial and reference probability measures. This flexibility is quite useful for the remeshing of surfaces. Given a triangulated surface $\mathcal{S}$ in $\R^3$, we can first map it to a parameter domain $D$ in $\R^2$, which we choose as a circular domain in our experiments. After that, we can generate a new mesh on the parameter domain $D$. The inverse map from the new mesh to the surface $\mathcal{S}$ gives a new mesh structure of $\mathcal{S}$. For example, Fig.~\ref{fig:e3_original} shows a triangulated surface and Fig.~\ref{fig:e3_initial} shows the initial harmonic map from the surface to a circular domain. Fig.~\ref{fig:e3_newmesh} shows a mesh structure generated on this domain using the DistMesh toolbox~\cite{persson2004simple}. 

\begin{figure}
\begin{subfigure}[t]{.33\textwidth}
\centering
\includegraphics[width=.9\linewidth]{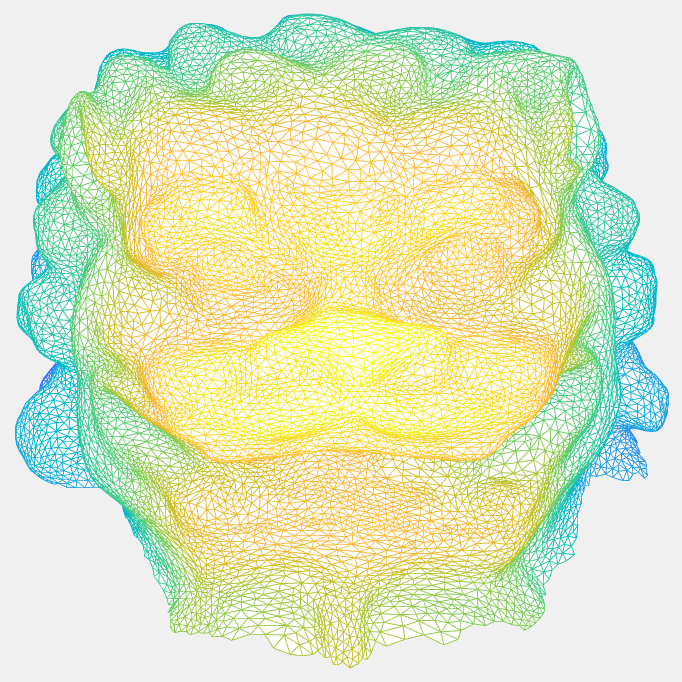}
\caption{The original mesh}
\label{fig:e3_original}
\end{subfigure}%
\begin{subfigure}[t]{.33\textwidth}
\centering
\includegraphics[width=.9\linewidth]{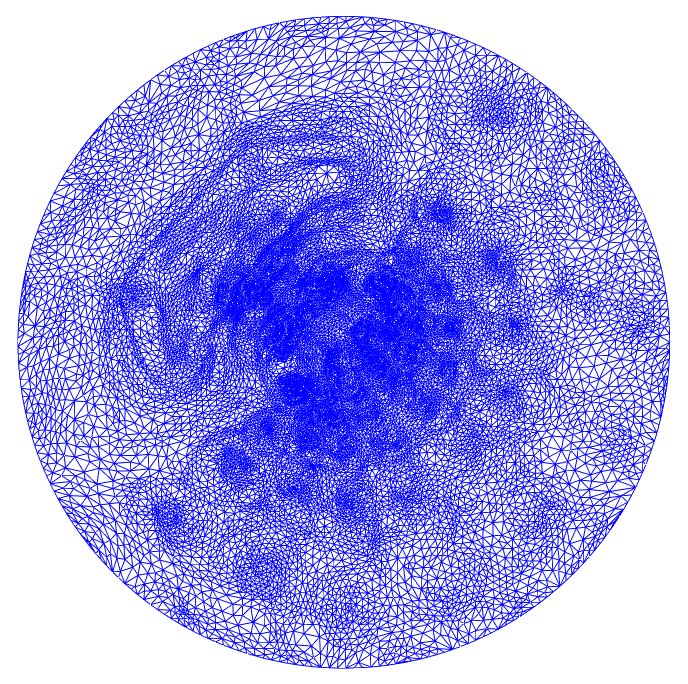}
\caption{The initial map}
\label{fig:e3_initial}
\end{subfigure}%
%\caption{A genus-3 example}
\begin{subfigure}[t]{.33\textwidth}
\centering
\includegraphics[width=.9\linewidth]{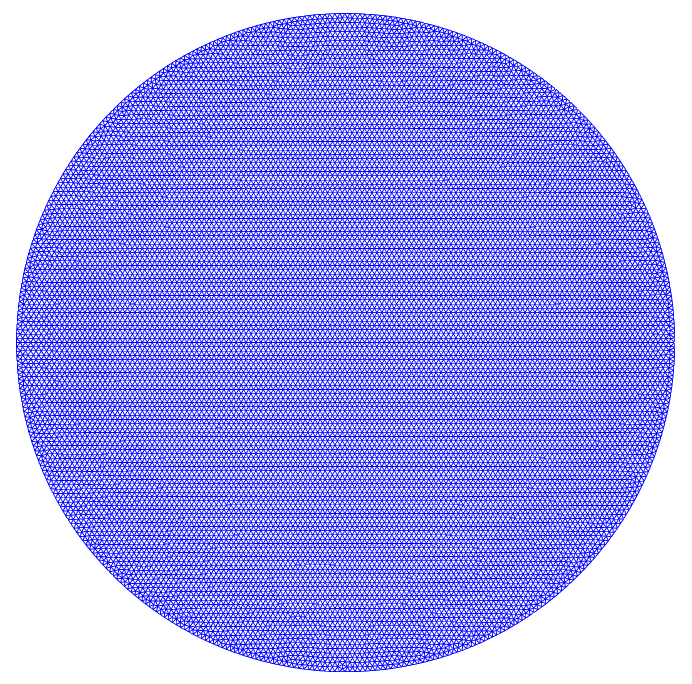}
\caption{A new mesh generated}
\label{fig:e3_newmesh}
\end{subfigure}%
\end{figure}

Fig.~\ref{fig:e3_results} shows some results of the optimization problem~\ref{eq:energy_measure_qc} with different initial and reference probability measures. To ensure consistency, the parameters $t_1$, $t_2$, and $t_3$ are fixed. As we can see, the vertex distributions in the parameter domain are quite different. As a result, the remeshing results given by the four maps are quite different from each other. This can be seen from Fig.~\ref{fig:e3_remeshes}. A general observation is that if the vertex distribution in some part of the parameter domain is sparse, then the distribution in the corresponding part of the remeshed surface will become dense. This can be seen by comparing the vertex distribution of Fig.~\ref{fig:e3_result2} and that of Fig.~\ref{fig:e3_remesh2}. Besides, since the Beltrami coefficient of the final maps are controlled by the second term of~\ref{eq:energy_measure_qc}, the triangles will not be greatly distorted when they are mapped back from~\ref{fig:e3_newmesh} to the surface.

\begin{figure}
\caption{Some maps by specifying different initial and reference probability measures}
\label{fig:e3_results}
\begin{subfigure}[t]{.25\textwidth}
\centering
\includegraphics[width=.9\linewidth]{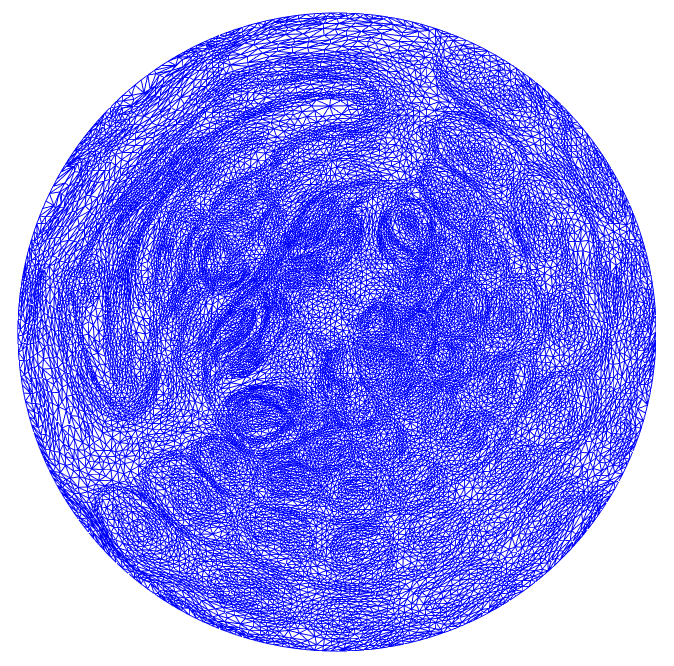}
\caption{Result 1}
\label{fig:e3_result1}
\end{subfigure}%
\begin{subfigure}[t]{.25\textwidth}
\centering
\includegraphics[width=.9\linewidth]{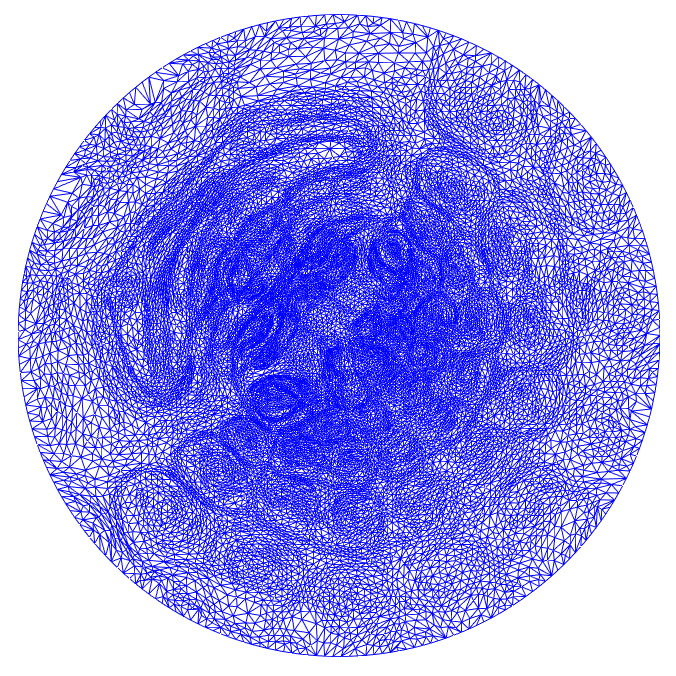}
\caption{Result 2}
\label{fig:e3_result2}
\end{subfigure}%
%\caption{A genus-3 example}
\begin{subfigure}[t]{.25\textwidth}
\centering
\includegraphics[width=.9\linewidth]{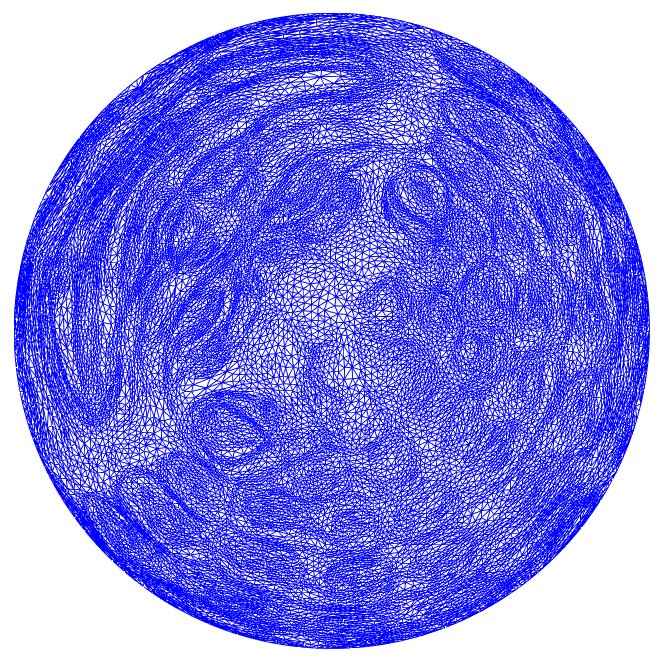}
\caption{Result 3}
\label{fig:e3_result3}
\end{subfigure}%
\begin{subfigure}[t]{.25\textwidth}
\centering
\includegraphics[width=.9\linewidth]{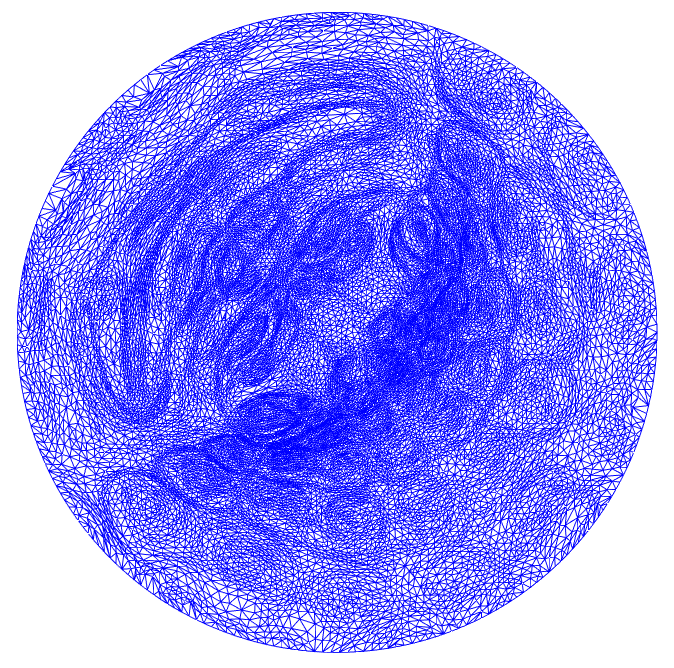}
\caption{Result 4}
\label{fig:e3_result4}
\end{subfigure}%
%\caption{A genus-3 example}
\end{figure}

\begin{figure}
\caption{The remeshing results given by the maps shown above}
\label{fig:e3_remeshes}
\begin{subfigure}[t]{.25\textwidth}
\centering
\includegraphics[width=.9\linewidth]{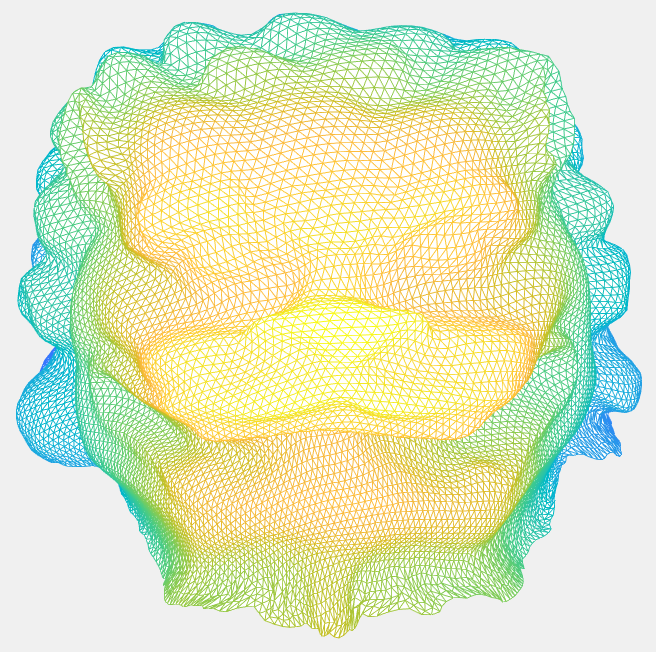}
\caption{Remeshed 1}
\label{fig:e3_remesh1}
\end{subfigure}%
\begin{subfigure}[t]{.25\textwidth}
\centering
\includegraphics[width=.9\linewidth]{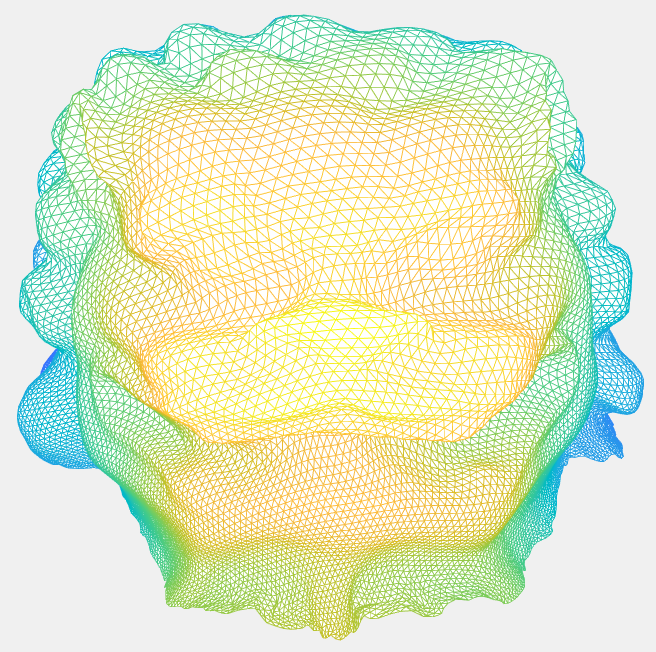}
\caption{Remeshed 2}
\label{fig:e3_remesh2}
\end{subfigure}%
%\caption{A genus-3 example}
\begin{subfigure}[t]{.25\textwidth}
\centering
\includegraphics[width=.9\linewidth]{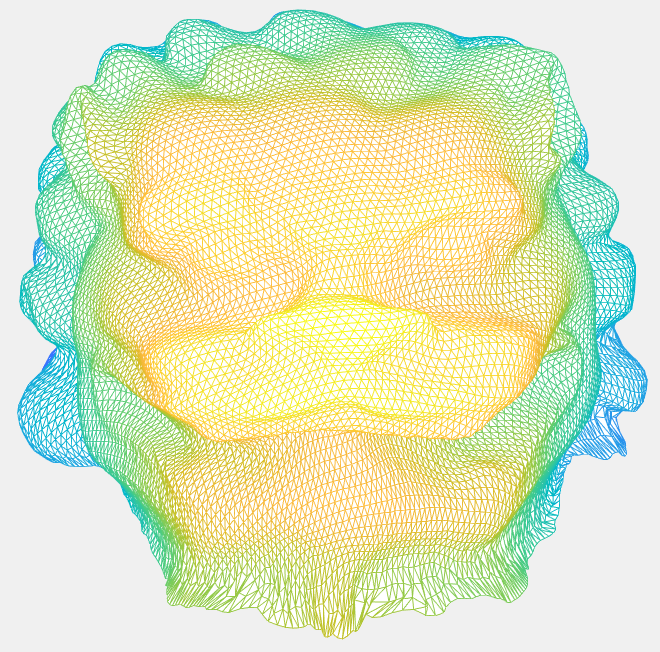}
\caption{Remeshed 3}
\label{fig:e3_remesh3}
\end{subfigure}%
\begin{subfigure}[t]{.25\textwidth}
\centering
\includegraphics[width=.9\linewidth]{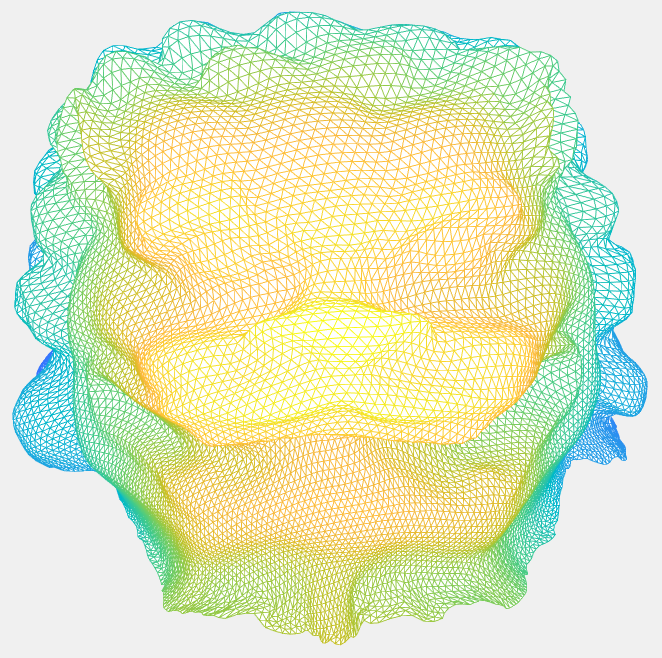}
\caption{Remeshed 4}
\label{fig:e3_remesh4}
\end{subfigure}%
%\caption{A genus-3 example}
\end{figure}

\bibliographystyle{siamplain}
\bibliography{references}
\end{document}